\documentclass[reqno]{amsart}
\usepackage{amsmath,amssymb}
\usepackage{graphicx,color}
\usepackage[all]{xy}
\theoremstyle{plain}
\newtheorem{introtheorem}{Theorem}

\newtheorem{theorem}{Theorem}[section]
\newtheorem{proposition}[theorem]{Proposition}
\newtheorem{lemma}[theorem]{Lemma}
\newtheorem{corollary}[theorem]{Corollary}

\theoremstyle{definition}
\newtheorem{definition}[theorem]{Definition}

\theoremstyle{remark}
\newtheorem{remark}[theorem]{Remark}

\def\E{{\mathcal E}}
\def\B{{\mathcal B}}

\def\Z{{\mathbb Z}}
\def\Q{{\mathbb Q}}

\def\I{{\mathcal I}}
\def\L{{\mathbb L}}
\def\T{{\mathbb T}}

\def\G{\mathcal G}

\def\cat0{\mathrm{cat}_0}

\def\aut{\mathrm{aut}}

\begin{document}

\title[]{Every  group is the group of self-homotopy equivalences of finite dimensional  CW-complex} 

\author{Mahmoud Benkhalifa}             
\email{mbenkhaifa@sharjah.ac.ae}
\address{Department of Mathematics. College of  Sciences, University of Sharjah. United Arab Emirates}


\begin{abstract}
We prove that any group $G$ occurs as  $\E(X)$, where $X$ is CW-complex of finite dimension and $\E(X)$
denotes its  group of self-homotopy equivalence. Thus, we  generalize a well-know theorem due to   Costoya and Viruel \cite{CV} asserting that any finite group   occurs as  $\E(X)$, where $X$ is rational elliptic space.
\end{abstract}

\maketitle

\section{introduction}
For a simply connected CW-complex $X$, we are interested in the group of 
self-homotopy equivalences $\E(X)$, and the so-called Kahn’s realisability problem of groups
\cite{k}. Namely, if a given group $G$ can occur as $\E(X)$ for some  space $X$.

For finite groups, in \cite{CV} Costoya and Viruel solved completely this problem by constructing
a rational elliptic space $X$ having formal dimension $n = 208 + 80\vert \mathcal G \vert$, where $\mathcal G$ is a certain
finite graph associated with $G$ and $\vert \mathcal G \vert$ denotes its order. The  space $X$ satisfies
$\pi_k(X) = 0$ for all $k \geq 120$. Later on in \cite{B1}, it is proven  that we
can realise the finite group $G$ by a non-elliptic space having formal dimension $n = 120$,
independently of the order of $\mathcal G$. 

It is worth mentioning that Kahn’s realisability problem has been solved for generic spaces in \cite{CMR} but is still open in the realm of CW-complexes for infinite groups. Thus, inspired by the ideas developed  in \cite{B0,Ben,CV1}, this paper aims to solve the quoted  problem  for   arbitrary groups and in the context of CW-complexes. 
\begin{introtheorem}
For any group $G$ and  any prime  $p>1114$, there exists a  CW-complex $X$ such that  $G\cong\E(X_{(p)})$, where $X_{(p)}$ is the $p$-localization of $X$. More precisely, we have
\begin{enumerate}
	\item $X$ is an $116$-connected, $2341$-dimensional, and of finite type if  $G$ is finite.
	\item $H_{i}(X,\Z_{(p)})$ is a free $\Z_{(p)}$-modules over a  basis which in bijection with $G$ for 
	$$i=691,q, \,\,\,\,\,\,\,2314\leq q\leq 2341.$$
	\item  $H_{i}(X,\Z_{(p)})\cong \Z_{(p)}$, for $i\in\{116,152,202,304,404,2314\}$.
\end{enumerate}	
\end{introtheorem}
\section{Anick's $Z_{(p)}$-local homotopy theory}
We prove our main theorem using standard tools  of  Anick’s differential graded Lie algebra framework for
$\Z_{(p)}$-local homotopy theory which we  refer to \cite{a1,a2,a3} for a general introduction to these techniques. We recall some of the notations here.
Let $\textbf{CW}^{k+1}
_m$ denote the category of $m$-connected, finite CW-complexes of dimension no greater than $k + 1$ with $m$-skeleton reduced to a point,  
$\textbf{CW}^{k+1}
_m(\Z_{(p)})$ denote the category obtained by $\Z_{(p)}$-localizing the CW-complexes in $\textbf{CW}^{k+1}
_m$ and  $\textbf{DGL}^{k}
_m(\Z_{(p)})$ denote the category  of the free differential graded Lie algebras (DGL for short) $(\L(W),\partial)$  in
which $W$ is a free graded  $\Z_{(p)}$-module satisfying $W_n=0$ for $n<m$ and $n>k$. 
\subsection{Homotopy in  $\textbf{DGL}^{k}_m(\Z_{(p)})$ (see [1, pp. 425-6])}
If $(\L(W),\delta)$ is an object of  $\textbf{DGL}^{k}
_m(\Z_{(p)})$, then we define the DGL $\L(W, sW, W'),D)$ as follows
 $$W\cong W',\,\,\,\,\,\,\,\,\,\,\,\,\,\,\,\,\,\,\,\,\,\,\,\,\,(sW)_{i}=W_{i-1}.$$ 
 The differential $D$ is given by
\begin{equation*}\label{g3}
	D(w)=\delta(w),\,\,\,\,\,\,\,\,\,\,\,\,\,\,\,\,\,\,\,\,\, D(sw)=w',\,\,\,\,\,\,\,\,\,\,\,\,\,\,\,\,\,\,\,\,\,D(w')=0.
\end{equation*}
Here $w'$ is the image of $w$ under the isomorphism $W\overset{\cong}{\to} W'$.
Define  $S$ as the derivation of degree +1 on $\L(W, sW, W')$ given by 
\begin{equation*}\label{51}
	S(w)=sw,\,\,\,\,\,\,\,\,\,\,\,\,\,\,\,\,\,\,\,\,\, S(sw)=S(w')=0.
\end{equation*}
A  homotopy   between two DGL-maps  $\alpha,\alpha':(\L(W),\delta)\to (\L(W),\delta)$  is a DGL-map
$$F \colon   (\L(W, sW, W'),D)\to (\L(W),\delta),$$
such that $F(w)=\alpha(w)$ and $F\circ e^{\theta}(w)=\alpha'(w)$,  where
$$e^{\theta}(w)=w+w'+\underset{n\geq 1}{\sum} \frac{1}{n!}(S\circ D)^{n}(w), \,\text{\,\,\, \,\,\,\,} \theta=D\circ S+S\circ D.$$

Subsequently,  we will need the following  lemma.
\begin{lemma}
	\label{l0} Let   $\alpha,\alpha' \colon   (\L( W_{\leq n}),\delta) \to (\L( W_{\leq n}),\delta)$  be two DGL-maps such that
	$$\alpha'(w)=\alpha(w)+y \text{ on }  W_{n} \,\,\,
	\text{ and }\,\,\,\, \alpha'=\alpha
	\text{ on }  W_{\leq n-1}.$$
	Assume that $y=\partial(z)$,
	where $z\in \L(W_{\leq n})$. Then $\alpha$ and $\alpha'$ are 
	homotopic.
\end{lemma}
\begin{proof}
	Define $F$ by setting
	\begin{eqnarray}
		F(w) \hspace{-2mm}&=&\hspace{-2mm} \alpha'(w), \,\,\,\,\,\,\,\,\,\,\,\,\,\,F(w')=-y \,\,\,\,\,\,\,\hbox{\ \ and \ \ }F(sw)=-z \hbox{ for } w \in W_{n},\nonumber \\
		F(w) \hspace{-2mm}&=&\hspace{-2mm} \alpha(w), \,\,\,\,\,\,\,\,\,\,\,\,\,\,\,\,F(w')=0 \,\,\,\,\,\,\,\,\,\,\,\,\hbox{\ \ and \ \ }F(sw)=0\,\,\,\, \hbox{ for } w \in W_{\leq n-1}\nonumber.
	\end{eqnarray}
	Then  $F$ is the needed  homotopy.
\end{proof}

Anick's work \cite{a3} asserts that if $k<\min(m+2p-3,mp-1)$, then the homotopy category of $\textbf{CW}^{k+1}
_m(\Z_{(p)})$   is equivalent to the
homotopy category of $\textbf{DGL}^{k}
_m(\Z_{(p)})$. Thus,  given a space $X$ in $\textbf{CW}^{k+1}_m(\Z_{(p)})$, Anick’s model recovers homotopy data via 
\begin{equation*}\label{1}
	\pi_{*}(X)\cong H_{*-1}((\L(W)),\partial)\,\,\,\,\,\text{ and }\,\,\,\,\,\,\,\,H_{*}(X,\Z_{(p)})\cong H_{*-1}(W,d),
\end{equation*}
where $d$ is the linear part of $\partial$ (see [1, Theorem 8.5]). Moreover, Anick’s theory directly
implies an identification of the form
\begin{equation}\label{2}
	\E(X)\cong \E(\L(W)), 
\end{equation}
where $\E(\L(W))$  is the group of differential graded Lie self-homotopy equivalences of $(\L(W),\partial)$
modulo the relation of homotopy in $\textbf{DGL}^{k}_m(\Z_{(p)})$ (see \cite{B1} for more details). 
\subsection{Strongly connected digraphs and theorem of J.  de Groot}
A   digraph (i.e. a
directed graph) $\G = (V(\G),E(\G))$,  where $V(\G)$ denotes the set of the
vertices of $\G$  and $E(\G)$ the set of its edges, is strongly connected if for every $v,u\in V(\G)$, there exists an integer $m\in \Bbb N$ and vertices $v=v_{0},v_{1},\dots,v_{m}=u$ such that $(v_i,v_{i+1})\in E(G)$ for all $i=0,1,\dots,m$.

\medskip
The following theorem, due to J. de Groot (\cite{G}, pp. 96), plays a central role in this paper.
\begin{theorem} \label{t5} Any group
	$G$ is isomorphic to the automorphism group of a  strongly connected digraph $\G$.
\end{theorem}
\begin{remark} 
	It should be noted that, in \cite{G},  de Groot is not considering strongly
	connected digraphs, but just graphs or digraphs.  But  any simple graph can
	be seen as a symmetric digraph, (\cite{N}, §1.1), if it is connected, then the associated digraph is
	strongly connected.
\end{remark}

\medskip

Now,  let us give an outlook of the steps  developed in the next sections  in order to prove the main result. 
\begin{itemize}
		\item To a given group $G$ corresponds a strongly connected digraph $\G$ such that $G\cong \aut(\G)$ according to Theorem \ref{t5}.
	\item To the  digraph $\G$, given above, we define a differential graded Lie $\Z_{(p)}$-algebra $\mathcal {L}_{}(\G,0)$.  We study the group $\E(\mathcal {L}_{}(\G,0))$ and we show that the sub-module of the cycles of degree 2313 is not trivial. 
	\item Next, we define a new DGL $\mathcal {L}_{}(\G,1)$  by adding generators in degree $2314$ to $\mathcal {L}_{}(\G,0)$  killing all the cycles of dimension 2313 in $\mathcal {L}_{}(\G,0)$. Likewise,  we study the group $\E(\mathcal {L}_{}(\G,1))$ and we show that the sub-module of the cycles of degree 2314 is not trivial.
	\item Similarly,   new DGLs $\mathcal {L}_{}(\G,s)$, $2\leq s\leq 24$, are  constructed  by   adding generators in degree $2314+s$ to $\mathcal {L}_{}(\G,s-1)$  killing all the cycles of dimension $2313+s$ in $\mathcal {L}_{}(\G,s-1)$.
	 Again,   we study the group $\E(\mathcal {L}_{}(\G,25))$ and we show that they are  cycles of degree 2339 in  $\mathcal {L}_{}(\G,25)$.
	\item Finally,  we construct the DGL $\mathcal {L}_{}(\G,26)$  by adding generators in degree $2340$ to $\mathcal {L}_{}(\G,25)$  killing all the cycles of dimension 2339 in $\mathcal {L}_{}(\G,25)$  and, crucially, we prove that    the sub-module of the cycles of degree 2340 is  trivial. This allows us to define an isomorphism of group from $\E(\mathcal {L}_{}(\G,26))$ to $\aut(\G)$. 
	\item Consequently,  using the Anick’s  $\Z_{(p)}$-local homotopy theory framework, where  $p>1114$  is  a prime, to  $\mathcal {L}(\G,26)$ corresponds  an object  $X$ in the category $\textbf{DGL}^{2341}
	_{115}(\Z_{(p)})$ satisfying 
	$$\mathcal{E}(X)\cong\mathcal{E}(\mathcal {L}(\G,26))\cong\aut(\G)\cong G.$$
\end{itemize}
\section{  Graded Lie $\Z_{(p)}$-algebras  associated with    strong connected digraphs}
\begin{definition}
	\label{d1}
Let  $p>1114$ be a prime.  For a given   strongly connected digraph $\G$,  with more than one vertex, we define the following DGL
	\begin{equation*}\label{30}
		\mathcal {L}_{}(\G,0)=\Big(\L(w_1,w_2,w_3,w_4,w_5,w_6,x_v,z_{(v,u)}\mid v\in V(\G), (v,u)\in E(\G)),\partial\Big).
	\end{equation*}
	The degrees of the generators  are as follows 
	$$\vert w_1\vert=115,\,\,\,\,\,\vert w_2\vert=151,\,\,\,\,\,\vert w_3\vert=201,\,\,\,\,\vert w_4\vert=303,\,\,\,\,\vert w_5\vert=403,\,\,\,\,\,\vert w_6\vert=2313$$ 
	\begin{equation}\label{bb5}
		\vert x_v\vert=690,\,\,\,\, \forall v\in V(\G),\,\,\,\,\,\,\vert z_{(v,u)}\vert=2337, \,\,\,\,\,\,\,\forall (v,u)\in E(\G).
	\end{equation} 
 The differential is given by
	\begin{eqnarray}
		\label{9}
		\partial(w_1)\hspace{-2mm}&=&\hspace{-2mm}\partial(w_2)=\partial(w_3)=0,\,\,\,\,\,\,\,\,\,\,\,\,\,\,\,\,\,\,\,\,\,\,\,\,\,\,\,\,\partial(x_v)=0,\,\,\,\,v\in V(\G),\\
		\partial(w_4)\hspace{-2mm}&=&\hspace{-2mm}[w_2,w_2],\,\,\,\,\,\,\,\,\partial(w_5)=[w_3,w_3],\nonumber\\
		\partial(w_6)\hspace{-2mm}&=&\hspace{-2mm}\Big[[w_2,w_3],(\operatorname{ad} w_3)^{9}(w_2)\Big]+(\operatorname{ad} w_1)^{2}\Big((\operatorname{ad} [w_4,w_2])^{4}([w_1,w_2])\Big),\nonumber\\
			\partial(z_{(v,u)})\hspace{-2mm}&=&\hspace{-2mm}(\operatorname{ad} x_v)^3([w_1,w_2])+(\operatorname{ad} w_1)^7([w_2,[x_v,x_u]])+(\operatorname{ad} w_1)^{19}(w_2)+Y+Z,\nonumber					
	\end{eqnarray}	
	where 	
	\begin{eqnarray}\label{12}
	Z\hspace{-2mm}&=&\hspace{-2mm}\big[(\operatorname{ad} w_1)^{5}(w_3),[(\operatorname{ad} w_2)^{3}(w_3),(\operatorname{ad} w_2)^{2}([w_3,w_5])]\big],\\
	Y\hspace{-2mm}&=&\hspace{-2mm}(\operatorname{ad} w_1)^{12}([[w_2,w_3],[w_3,w_5]])\nonumber.
	\end{eqnarray}
\end{definition}
Recall that the iterated Lie bracket of length $k + 1$  is defined by
$$(\operatorname{ad} x)^{k}(y)=[x,[x,[\dots,[x,y]\dots], $$
where $x$ is involved $k$ times. For the sake of simplicity,  let's denote
\begin{equation*}\label{22}
	\tilde{\mathcal {L}}_{}=\big(\L(w_1,w_2,w_3,w_4,w_5,x_v\mid v\in V(\G)),\partial\big).
\end{equation*}

\begin{remark}
	\label{rrr2}
The prime number $p$ and the degrees of the elements in the DGL $\Z_{(p)}$-algebra $\mathcal {L}_{}(\G,0)$  have been meticulously selected to align with Anick's framework. This choice is made with consideration of the  constraint imposed by the dimension and connectivity of the DGL $\mathcal {L}_{}(\G,0)$, namely the relation $k<\min(m+2p-3,mp-1)$, where $k$ is the dimension and $m$ is the connectivity of $\mathcal {L}_{}(\G,0)$, mentioned in section 2. Note that as by (\ref{bb5}), we have $m=115$ and $k=2337$, it follows that the prime $p$ can be chosen as  $p>1114$. 
\medskip

Later, we will make use of the following  lemmas.
\end{remark}

\begin{lemma}	
	\label{l5}
	For every $v\in V(\G)$,  the following cycles 	are not boundaries in $\tilde{\mathcal {L}}_{}$, $$[x_v,[x_u,[x_t,[w_2,w_1]]]],\,\,\,(\operatorname{ad} x_v)^3([w_1,w_2]),\,\,\,\,(\operatorname{ad} w_1)^7([w_2,[x_v,x_u]])),\,\,\,(\operatorname{ad} w_1)^{19}(w_2).$$
\end{lemma}
\begin{proof}
	Since   the only non zero differentials are $\partial(w_3)=[w_2,w_2]$, $\partial(w_5)=[w_3,w_3]$ and none of the given  cycles contain $w_2$ or $w_3$ twice, it follows that the given brackets  cannot be boundaries.
\end{proof}	
\begin{lemma}\label{l1}
	For every $s,s'\in V(\G)$ such that $s\neq s'$, the following brackets 
	$$[x_s,[x_s,[x_{s'},[w_1,w_2]]]],\,\,\,\,[x_s,[x_{s'},[x_{s},[w_1,w_2]]]],\,\,\,\,[x_{s'},[x_s,[x_{s},[w_1,w_2]]]],$$
	are linearly independent. Moreover, any linear combination of those brackets is not a boundary  in $\tilde{\mathcal {L}}_{}$.
\end{lemma}
\begin{proof}
		First, 	the condition $p>1114$ ensures that $\mathcal {L}_{}(\G,0)$ is an object of the category   $\textbf{DGL}^{2341}_{115}(\Z_{(p)})$  allowing us to use  Anick's framework in \cite{a2}. 
		
		\noindent Next, let  $\tilde{\mathcal {T}}=\T(w_1,w_2,w_3,w_4,w_5,\{x_v\}_{ v\in V(\G)})$ denotes the  universal algebra of $\tilde{\mathcal {L}}_{}$. Recall that $\tilde{\mathcal {T}}$ can be considered as a graded Lie algebra by defining
	\begin{equation}\label{8}
		[B,C]=BC-(-1)^{\vert B\vert\vert C\vert}	CB,\,\,\,\,\,\,\,\,\,\,\,\,\,\,\,\,\,\,\,\,\,\,\,\,B,C\in \tilde{\mathcal {T}}.
	\end{equation}
Note that   the differential in $\tilde{\mathcal {T}}$ satisfies
\begin{equation}\label{a45}
	\partial(w_4)=2w_2^2,\,\,\,\,\,\,\,\,\,\,\,\,\,\,\,\,\,\,\partial(w_3)=2w_3^2.
\end{equation}
	Next, set $A=[w_1,w_2]$,  if
	$$\mu_1[x_s,[x_s,[x_{s'},A]]+\mu_2[x_s,[x_{s'},[x_{s},A]]]+\mu_3[x_{s'},[x_s,[x_{s},A]]]=0,\,\,\,\,\,\,\,\mu_1,\mu_2,\mu_3\in \Z_{(p)},$$
	then,  expanding  the bracket $\mu_1[x_s,[x_s,[x_{s'},A]]]$ in $\tilde{\mathcal {T}}$ by using (\ref{8}), we get the monomial $\mu_1x^2_sx_{s'}A$. Similarly, expanding  $[x_s,[x_{s'},[x_{s},A]]]$ and $[x_{s'},[x_{s},[x_{s},A]]]$, we obtain 
	\begin{eqnarray}\label{220}
		[x_s,[x_{s'},[x_{s},A]]]\hspace{-2mm}&=&\hspace{-2mm}x_sx_{s'}x_sA-x_sx_{s'}Ax_s-x_s^2Ax_{s'}+x_sAx_{s}x_{s'}-\nonumber\\&&x_{s'}x_{s}Ax_s+x_{s'}x_{s}Ax_s+x_sAx_{s'}x_s-Ax_sx_{s'}x_s,
	\end{eqnarray}
	\begin{equation}\label{221}
		[x_{s'},[x_{s},[x_{s},A]]]=x_{s'}x^2_sA-2x_{s'}x_{s}Ax_s-x_{s'}Ax_s^2-
		x_{s}^2Ax_{s'}-2x_{s}Ax_{s}x_{s'}-Ax^2_{s}x_{s'}.
	\end{equation}
	Therefore $\mu_1x^2_sx_{s'}A$ does not appear in (\ref{220}) and in (\ref{221}). As a result $\mu_1=0$. 
	
	\noindent Likewise, as the expression    $\mu_2x_sx_{s'}x_sw_1w_2$ does not appear in  (\ref{221}), it follows that $\mu_2=0$. Consequently $\mu_3=0$.
	\medskip
	
	Finally, since the monomials $\mu_1x^2_sx_{s'}w_2w_1$, $\mu_2x_sx_{s'}x_sw_1w_2$ and $\mu_3x_{s'}x_{s}x_sw_1w_2  $ cannot   be reached by the differential  $\partial$ according to (\ref{a45}), it follows that any linear combination of the given  brackets is not a boundary  in $\tilde{\mathcal {L}}$.
\end{proof}
\begin{lemma}\label{l9}
	For every $v\in V(\G)$, the  brackets  $Y$ and $Z$, given in \rm{(\ref{12})}, 
	are linearly independent. Moreover, any linear combination of $Y$ and $Z$ is not a boundary in $\tilde{\mathcal {L}}$.
\end{lemma}
\begin{proof}
First, recall that $$Y=(\operatorname{ad} w_1)^{12}([[w_2,w_3],[w_3,w_5]]),\,\,Z=\big[(\operatorname{ad} w_1)^{5}(w_3),[(\operatorname{ad} w_2)^{3}(w_3),(\operatorname{ad} w_2)^{2}([w_3,w_5])]\big].$$
	As  $Y$ contains exactly 2 generators $w_{3}$, $Z$ has 3, we derive that  $Y$ and $Z$ are linearly independent.
	
	\medskip
	Next, we claim that any linear combination $\gamma_1Y+\gamma_2Z$, where $\gamma_1,\gamma_2\in \Z_{(p)}$,  cannot be a boundary.  Indeed,   using the same argument given in the previous Lemmas,    expanding the three brackets in the universal algebra $\tilde{\mathcal {T}}$,  we get the following monomials 
	$$\gamma_1w_1^{12}w_2w_3^2w_5,\,\,\,\,\,\,\,\,\gamma_2w_1^{5}w_3 w_2^{3}w_3w_2^{2}w_3w_5,$$
	and none of them cannot   be reached by the differential  $\partial$ according to (\ref{a45}).
\end{proof}
\begin{remark}
	\label{ra1}	Using the same arguments as in the proof of Lemmas (\ref{l1}) and (\ref{l9}), we can easily prove that any linear combination of  the brackets 
	$$\Big[[w_2,w_3],(\operatorname{ad} w_3)^{9}(w_2)\Big],\,\,\,\,\,(\operatorname{ad} w_1)^{2}\Big((\operatorname{ad} [w_4,w_2])^{4}([w_1,w_2])\Big),$$
	given in the relation (\ref{9}), is not boundary in $\tilde{\mathcal {L}}$.
\end{remark}
\section{Main result}
\subsection{Studying the group 	$\E(\mathcal {L}_{}(\G,0))$}
First, let $\mathcal{L}_k(\mathcal{G}, 0)$ denote the sub-module consisting of elements of degree $k$. Next, we initiate the process with the following Lemmas.
\begin{lemma}	
	\label{b0} The set  of the decomposable elements in the $\mathcal {L}_{690}(\G,0)$  is empty.
\end{lemma}	
\begin{proof}
	Assume  $\Theta\in\mathcal {L}_{690}(\G,0)$ is a decomposable element. 
	then  we can write
	$$\vert \Theta\vert =a_1\vert w_1\vert +a_2\vert w_2\vert+a_3\vert w_3\vert+a_4\vert w_4\vert+a_5\vert w_5\vert=690,$$	
and 	by using the relation (\ref{bb5})	we obtain 
	\begin{equation}\label{bb4}
		115a_1 +151a_2+201a_3+303a_4+403a_5=690.
	\end{equation}
	The equation (\ref{bb4}),	where $a_1,a_2,a_3,a_4,a_5\in \{0,1,2,\dots\}$ are  unknown, is called a  Frobenius equation and can be solved by WOLFRAM\footnote{ See the link below for more details \\ https://reference.wolfram.com/language/tutorial/Frobenius.html} software using the following code
	\begin{center}
		"FrobeniusSolve[\{115, 151, 201, 303, 403\}, 690]".
	\end{center}
	The only solution to (\ref{bb4}) is $a_1=6,a_2=a_3=a_4=a_2=a_5=0$. That means  $\Theta$ is a bracket formed using 6 generators $w_1$ which is impossible.
\end{proof}
By virtue of Lemma \ref{b0} and for degree reasons, it follows that for every  $[\alpha]\in \E(\mathcal{L}(G,0))$, we can write
\begin{eqnarray}
	\label{11}
	\alpha(w_1)\hspace{-2mm}&=&\hspace{-2mm}\beta w_1,\,\,\,\,\,\alpha(w_2)=\lambda w_2,\,\,\,\,\alpha(w_3)=\gamma w_3,\,\,\,\alpha(w_4)=q w_4,\,\,\,\alpha(w_5)=r w_5,\nonumber\\	
	\alpha(x_v)\hspace{-2mm}&=&\hspace{-2mm}\sum_{s\in V(\G)}a_{(v,s)}x_s,\nonumber\\
	\alpha(w_6)\hspace{-2mm}&=&\hspace{-2mm}\mu w_6+F,\nonumber\\
\alpha(z_{(v,u)})	\hspace{-2mm}&=&\hspace{-2mm}\sum_{(r,s)\in E(\G)}\rho_{(v,u),(r,s)}z_{(r,s)}+B_{(v,u)},	
\end{eqnarray}
where all  the coefficients belong to $\Z_{(p)}$ and where $B_{(v,u)}$ and $F$ are decomposable elements in $\mathcal{L}_{2337}(\G,0)$ and $\mathcal{L}_{2313}(\G,0)$ respectively. 
\begin{remark}	
	\label{r0}
 Almost every coefficients $\rho_{(v,u),(r,s)}$ and $a_{(v,s)}$ is zero. Moreover,  as $\alpha$ is a  homotopy equivalence, then its induces an isomorphism on the indecomposables, therefore   $\beta,\lambda,q,\gamma,r\neq 0$ and at least one of the coefficients $a_{(v,s)}$ is not zero as well as,   at least one of the coefficients $\rho_{(v,u),(r,s)}$ is not zero.
\end{remark} 
\begin{lemma}	
	\label{p1}
	If $[\alpha]\in \E(\mathcal{L}(G,0))$, then 
$q=\lambda^2$ and $r=\gamma^2$.
\end{lemma} 
\begin{proof}	
Since $\partial(\alpha(w_5))=\alpha(\partial(w_5))$ and $\partial(\alpha(w_3))=\alpha(\partial(w_3))$, it follows that 
	$$\partial(\alpha(w_5))=r[w_3,w_3],\,\,\,\,\,\,\,\,\,\alpha(\partial(w_5))=[\alpha(w_3),\alpha(w_3)]=\gamma^2[w_3,w_3],$$
	$$\partial(\alpha(w_3))=q[w_2,w_2],\,\,\,\,\,\,\,\,\,\alpha(\partial(w_3))=[\alpha(w_2),\alpha(w_2)]=\lambda^2[w_2,w_2].$$
	Consequently, $q=\lambda^2$ and $r=\gamma^2.$
\end{proof} 
\begin{proposition}		\label{p4}
	Let $[\alpha])\in \mathcal{L}(G,0)$.  There exists unique  $\phi\in\aut(\G)$ such that 
	\begin{eqnarray}
		\label{182}
	\alpha(z_{(v,u)})\hspace{-2mm}&=&\hspace{-2mm}z_{(\phi(v),\phi(u))}+B_{(v,u)},\,\,\,\,\,\,\,\,\,\,\,\,\forall (v,u)\in E(\G),\\
	\alpha(w_6)\hspace{-2mm}&=&\hspace{-2mm} w_6+F,\nonumber\\	
		\alpha_{}(x_v)\hspace{-2mm}&=&\hspace{-2mm}x_{\phi(v)},\,\,\,\,\,\,\,\,\,\,\,\,\,\,\,\,\,\,\,\,\,\,\,\,\,\,\,\,\,\,\,\,\,\,\,\,\,\,\,\,\,\,\,\,\,\,\,\,\,\forall v\in V(\G),\nonumber\\
		\alpha_{}(w_i)	\hspace{-2mm}&=&\hspace{-2mm}w_i,\,\,\,\,\,\,\,\,\,\,\,\,\,\,\,\,\,\,\,\,\,\,\,\,\,\,\,\,\,\,\,\,\,\,\,\,\,\,\,\,\,\,\,\,\,\,\,\,\,\,\,\,\,\,\,\,i=1,2,3,4,5.\nonumber	
	\end{eqnarray}
Moreover,  $B_{(v,u)}$ and $F$ are cycles in $\mathcal{L}_{2337}(\G,0)$ and $\mathcal{L}_{2313}(\G,0)$ respectively. 
\end{proposition}
\begin{proof}	
	Notice    that the strong connectivity of the graph implies that for every $v \in V (\G)$, $v$
	is the starting vertex of an edge $(v, w) \in E(\G)$. Therefore the coefficients in (\ref{11}) can be entirely determined by the relation $\alpha\circ\partial=\partial\circ\alpha$. 
	Indeed, first we have 
\begin{eqnarray}
	\label{x13}
	\alpha(\partial(w_6))\hspace{-3mm}&=&\hspace{-3mm}\big[[\alpha(w_2)\alpha(,w_3)],(\operatorname{ad} \alpha(w_3))^{9}(\alpha(w_2))\big]+(\operatorname{ad} \alpha(w_1))^{2}\Big((\operatorname{ad} [\alpha(w_4),\alpha(w_2)])^{4}([\alpha(w_1),\alpha(w_2)])\Big)\nonumber\\
	\hspace{-3mm}&=&\hspace{-3mm}\lambda^2\gamma^{10}\big[[w_2,w_3],(\operatorname{ad} w_3)^{9}(w_2)\big]+\beta^{3}\lambda^{13}(\operatorname{ad} w_1)^{2}\Big((\operatorname{ad} [w_4,w_2])^{4}([w_1,w_2])\Big)\nonumber\\
	\partial( \alpha(w_6))	)\hspace{-3mm}&=&\hspace{-3mm}\mu \partial(w_6)+ \partial(F)=\mu\big[[w_2,w_3],(\operatorname{ad} w_3)^{9}(w_2)\big]+\mu(\operatorname{ad} w_1)^{2}\Big((\operatorname{ad} [w_4,w_2])^{4}([w_1,w_2])\Big)+ \partial(F)\nonumber.		
\end{eqnarray}
Since $\alpha(\partial(w_6))=\partial( \alpha(w_6))	)$ and $F$ is decomposable element  in $\mathcal{L}_{2337}(\G,0)$, we deduce that 
\begin{equation}
	\label{x3}
	\mu=\lambda^2\gamma^{10}=\beta^{3}\lambda^{13},\,\,\,\,\,\,\,\,\,\,\,\,\,\,\partial(F)=0
\end{equation}
Next, we have
	\begin{eqnarray}
		\label{13}
		\alpha(\partial(z_{(v,u)})))\hspace{-2mm}&=&\hspace{-2mm}\big(\operatorname{ad} 	(\alpha(x_v))\big)^3([	\alpha(w_1),	\alpha(w_2)])+\big(\operatorname{ad} (	\alpha(w_1))\big)^7([	\alpha(w_2),[	\alpha(x_v),	\alpha(x_u)]])+\nonumber\\
		&&+\big(\operatorname{ad} (\alpha(w_1))\big)^{19}(\alpha(w_2))	+	\alpha(Y)+	\alpha(Z),\nonumber\\
		\partial( \alpha(z_{(v,u)}))	)\hspace{-2mm}&=&\hspace{-2mm}\sum_{(r,s)\in E(\G)}\rho_{(v,u),(r,s)}\partial(z_{(r,s)})+\partial(B_{(v,u)}),	\\	
		\hspace{-2mm}&=&\sum_{(r,s)\in E(\G)}\rho_{(v,u),(r,s)}\Big((\operatorname{ad} x_r)^3([w_1,w_2])+(\operatorname{ad} w_1)^7([w_2,[x_r,x_s]])+\nonumber\\
		&&(\operatorname{ad} w_1)^{19}(w_2)+Y+Z\Big)+\partial(B_{(v,u)})\nonumber.		
	\end{eqnarray}
	where $Y,Z$ are given in (\ref{12}). 
	Next, recall that from  the relations (\ref{11}), we have
	\begin{equation}\label{23}
	\alpha(w_1)=\beta w_1,\,\,\,\,\,\,\,\,\,\,\alpha(w_2)=\lambda w_2,\,\,\,\,\,\,\,\,\,	\alpha(x_v)=\sum_{s\in V(\G)}a_{(v,s)}x_s,
	\end{equation}
	where almost every coefficients  $a_{(v,s)}$ is zero  and at least one of them  is not zero. Expanding the expression  
	\begin{equation*}\label{14}
		(\operatorname{ad} 	\alpha(x_v))^3([\alpha(w_1),	\alpha(w_2)]),
	\end{equation*}
 using (\ref{23}),  we get to the following brackets  
	\begin{equation}\label{244}
		\beta\lambda a^2_{(v,s)}a_{(v,s')}[x_s,[x_s,[x_{s'},[w_2,w_1]]]],\,\,\,\,\,\,\,\,\,\,\,\, v,s,s'\in V(\G),
	\end{equation}
	$$\beta\lambda a^2_{(v,s)}a_{(v,s')}[x_s,[x_{s'},[x_{s},[w_2,w_1]]]],\,\,\,\,\,\,\,\,\,\,\,\, v,s,s'\in V(\G),$$
	$$\beta\lambda a^2_{(v,s)}a_{(v,s')}[x_{s'},[x_s,[x_{s},[w_2,w_1]]]],\,\,\,\,\,\,\,\,\,\,\,\, v,s,s'\in V(\G).$$
 	These brackets are proven to be linearly independent according to Lemma 3.3.
 	
 	\noindent However, none of the brackets in the expression (\ref{13}), giving  $\partial( \alpha(z_{(v,u)}))$, is formed using  three generators $x_s,x_s,x_{s'}$ with $s\neq s'$. Moreover, by Lemma \ref{l1},  the expressions (\ref{244}) as well as the following expression 
	$$	\beta\lambda a^2_{(v,s)}a_{(v,s')}\Big([x_s,[x_s,[x_{s'},[w_1,w_2]]]]+[x_s,[x_{s'},[x_{s},[w_1,w_2]]]]+[x_{s'},[x_s,[x_{s},[w_1,w_2]]]]\Big),$$
	are neither trivial nor a boundaries.  
	
	\noindent Now, by the formula $\alpha(\partial(z_{(v,u)})))=\partial( \alpha(z_{(v,u)})))$, we deduce that  all of the coefficients 	$\beta\lambda a^2_{(v,s)}a_{(v,s')}$  are  nil.  Since $\beta,\lambda\neq 0$, it follows that  only one  coefficient among $a_{(v,s)}$, where $s\in V(\G)$,    is not zero. Let us denote it by $a_{(v,t_v)}$. As a result, the formula (\ref{23}) becomes $\alpha(x_v)=a_{(v,t_v)}x_t.$ Thus, there is  a  unique  vertex $t\in V(\G)$ such that  $\alpha(x_v)=a_{(v,t_v)}x_t$.

	Consequently, on  the one hand and going back to (\ref{12}) and (\ref{11}),  we deduce that
	$$\alpha(Y)=\beta^{12}\lambda\gamma^4Y,\,\,\,\,\,\,\,\,\,\,\,\,\,\,\,\alpha(Z)=\beta^{5}\lambda^5\gamma^5Z.$$	
	On the other hand,  the formulas  (\ref{13}) become
	\begin{eqnarray}
		\label{133}
		\alpha(\partial(z_{(v,u)}))\hspace{-2mm}&=&\hspace{-2mm}\beta\lambda a^3_{(v,t_v)}\big(\operatorname{ad}(x_t))^3([w_1,w_2])+\beta^{7}\lambda a_{(v,t_v)}a_{(u,t_u)}\big(\operatorname{ad} (w_1)\big)^7([w_2,[x_t,x_{t'}]])\nonumber\\
		&&+\beta^{19}\lambda\big(\operatorname{ad} (w_1))^{19}(w_2)	+\beta^{12}\lambda\gamma^4Y	+	\beta^{5}\lambda^5\gamma^5Z,\nonumber\\
		\partial( \alpha(z_{(v,u)}))	\hspace{-2mm}&=&\hspace{-2mm}\sum_{(r,s)\in E(\G)}\rho_{(v,u),(r,s)}\Big((\operatorname{ad} x_r)^3([w_1,w_2])+(\operatorname{ad} w_1)^7([w_2,[x_r,x_s]])+\nonumber\\
		&&(\operatorname{ad} w_1)^{19}(w_2)+Y+Z\Big)+\partial(B_{(v,u)}).	
	\end{eqnarray}
	Likewise, due to Lemmas \ref{l5} and \ref{l9},  all the brackets in $\partial( \alpha(z_{(v,u)}))-\partial(B_{(v,u)})$ and $\alpha(\partial(z_{(v,u)}))$
	are not boundaries, and by  comparing the coefficients in the formula
	\begin{equation}\label{b12}
		\partial( \alpha(z_{(v,u)}))-\alpha(\partial(z_{(v,u)}))=0,
	\end{equation}
	we deduce that all the coefficients $\rho_{(v,u),(r,s)}=$ are zero except   $\rho_{(v,u),(t_v,t_u)}\neq 0$ which satisfies the following equations
	\begin{equation}\label{333}
		\rho_{(v,u),(t_v,t_u)}=\beta\lambda a^3_{(v,t_v)}=\beta^{7}\lambda a_{(v,t_v)}a_{(u,t_u)}=\beta^{19}\lambda=\beta^{12}\lambda\gamma^4=\beta^{5}\lambda^5\gamma^5.
	\end{equation}
	From $\beta^{19}\lambda=\beta^{12}\lambda\gamma^4=\beta^{5}\lambda^5\gamma^5$, we deduce that
	$$\beta^{7}=\gamma^4,\,\,\,\,\,\,\beta^{12}=\lambda^4\gamma^5,\,\,\,\,\,\beta^{7}=\lambda^4\gamma\Longrightarrow \gamma^{3}=\lambda^4,$$
	therefore,
	$$(\beta^{7})^{12}=\gamma^{48}=(\beta^{12})^7=\lambda^{28}\gamma^{35}\Longrightarrow \gamma^{13}=\lambda^{28}=(\gamma^{3})^7=\gamma^{21}.$$
	It follows that $\gamma^{8}=1$. As $\gamma^{3}=\lambda^4$ and $\beta^{7}=\gamma^4$, we deduce that $\beta=\gamma=1$.  As a result,  the relation (\ref{x3}) becomes   $\lambda^2=\lambda^{13}$ implying that $\lambda=1$. 
	
	\noindent Now, the relations (\ref{333}) become
	$$\rho_{(v,u),(t_v,t_u)}= a^3_{(v,t_v)}=a_{(v,t_v)}a_{(u,t_u)}= 1.$$
	Consequently,  we get 
	$$\rho_{(v,u),(t_v,t_u)}=\beta=\lambda= a_{(v,t_v)}=a_{(u,t_u)}=\gamma=1,$$
	and from  Lemma \ref{p1}, it follows that $q=r=1$. Thus, the formulas (\ref{133}) become
	\begin{eqnarray}
		\label{1333}
		\alpha(\partial(z_{(v,u)}))\hspace{-2mm}&=&\hspace{-2mm}\big(\operatorname{ad}(x_t))^3([w_1,w_2])+\big(\operatorname{ad} (w_1)\big)^7([w_2,[x_t,x_{t'}]])+\big(\operatorname{ad} (w_1))^{19}(w_2)	+\nonumber\\&&Y	+	Z,\nonumber\\
		\partial( \alpha(z_{(v,u)}))	\hspace{-2mm}&=&\hspace{-2mm}\operatorname{ad}(x_t))^3([w_1,w_2])+\big(\operatorname{ad} (w_1)\big)^7([w_2,[x_t,x_{t'}]])+\big(\operatorname{ad} (w_1))^{19}(w_2)	+\nonumber\\
		&&Y	+	Z+(\operatorname{ad} w_1)^{19}(w_2)+Y+Z+\partial(B_{(v,u)})\nonumber.	
	\end{eqnarray} 
and from (\ref{b12})	, it follows that $\partial(B_{(v,u)})=0$.
	
	Thus,  going back the formulas  (\ref{11}), we have  proved that for  every $v\in V(\G)$,  there is  a  unique  vertex $t_v\in V(\G)$ and  for  every $(v,u)\in E(\G)$,  there is  a  unique  edge $(t_v,t_u)\in E(\G)$ such that
	\begin{equation*}\label{36} \alpha(z_{(v,u)})=z_{(t_v,t_u)}+B_{(v,u)},\,\,\,\,\alpha(x_v)=x_{t_v},\,\,\,\,\alpha(w_6)= w_6+F,\,\,\,\,\alpha(w_i)=w_i,\,\,\,\,1,2,3,4,5.
	\end{equation*}
with  both $B_{(v,u)}$ and $F$ being cycle as desired.

 Thus,  define $\phi:\G\to \G$, by $\phi(v)=t_v$, $\phi((v,u))=(t_v,t_u)$, we obtain (\ref{182}).
\end{proof}
\begin{remark}
	\label{m0}
	Let 	$\mathcal {L}_{}(\G,27)=\mathcal {L}_{}(\G,0)\oplus\Big(\L(h, h\in \I),\partial\Big)$ be the DGL obtained from $\mathcal {L}_{}(\G,0)$ by adding generators $ h\in \I$, where $2314\leq\vert h\vert\leq 2340$.
	It is easy to see that $\mathcal {L}_{}(\G,27)$ does not contain any bracket of degree $d$,  where $2313\leq d\leq 2340$, formed by  using  three  generators from  the set  $\{x_v\}_{v\in V(\G)}$. Indeed;   if a bracket  contains three generators from    $\{x_v\}_{v\in V(\G)}$, it follows that the sum of the degrees of  the other generators forming   this bracket is $d-3\times 690$.  But we have $243\leq d-3\times 690\leq 270$ and we know that no element in $\mathcal{L} _{}(\G,27)$ has degree between 243 and  270. Consequently, for  $0\leq n\leq 27$,  if the sub-module  $Z_{n}(\mathcal{L}(\G,27))$ of the cycles of degree $2313+n$ in $\mathcal {L}_{}(\G,27)$ is not trivial,  then  we can choose a Hall basis for  $Z_{n}(\mathcal{L}(\G,27))$ the following  set  
	\begin{equation}\label{z9}
	\B_{n}=\Big\{y_{n,1},\dots,y_{n,m_{n}},y_{s,v}, y_{n,\{v',u\}}; v,v',u\in V(\G)\Big\},
	\end{equation}
where
	\begin{enumerate}
		\item $y_{n,1},\dots,y_{n,m_{n}}\in\L(w_1,w_2,w_3,w_4,w_5)$.
		\item $y_{n,v}$ is formed by  using  a generator from  the set $\{x_v\}_{v\in V(\G)}$ and $w_1,w_2,w_3,w_4,w_5$.
		\item $y_{n,\{v',u\}}$ is formed by  using  two generators from   $\{x_v\}_{v\in V(\G)}$ and $w_1,w_2,w_3,w_4,w_5$.
	\end{enumerate}
Furthermore,  based on   Proposition \ref{p4} and  taking into account that  $y_{s,v}$ and  $y_{s,\{v',u\}}$  contain   elements from  $\{x_v\}_{v\in V(\G)}$,	if $[\alpha]\in \mathcal{E}(\mathcal {L}(\G,27))$, then  there is a unique  $\phi\in\aut(\G)$ such that for every $v,v',u\in V(\G)$ and $0\leq n\leq 27$, we have
$$\alpha(y_{n,v})=y_{n,\phi(v)},\,\,\,\,\,\,\,\,\,\,\,\,\,\,\,\,\,\,\,\,\,\,
\alpha(y_{n,\{v',u\}})=y_{n,\{\phi(v'),\phi(u)\}}$$
\end{remark}
\begin{lemma}	
	\label{xb9} The sub-module $Z_{2313}(\mathcal{L}(\G,0))$ of the cycles of degree $2313$ is not trivial. 
\end{lemma}	
\begin{proof}
First, Indeed, for instance,	it is easy to check that the following bracket
$$[(\operatorname{ad} w_{3})^{6}(w_2),(\operatorname{ad} w_{2})^{5}(w_3)]$$
is a  cycle of degree $2313$. 
\end{proof}
\begin{remark}
	\label{xc2}
		Due to Proposition \ref{p4}, we know that  $F$, given in (\ref{182}), is a cycle.  Hence, by Remark \rm{\ref{m0}}, we can write  $F$ as a linear combination of the elements of $\B_{0}$.
\end{remark}
The forthcoming lemma, presented in a broad context, is poised to play a pivotal role in the forthcoming developments.
\begin{lemma}
	\label{xz1}
	Let $\G$ be a graph and $\phi\in\aut(\G)$. Assume  $(\mathbb {L}_{}(W),\partial)$ is a DGL such that  the sub-module $Z_{p}(\L(W))$ has for a Hall basis the set $\{y_{1},\dots,y_{m},y_v,y_{\{v',u\}}\}_{v,v',u\in V(\G)}$. 
	Let  $\mathbb {L}_{}(U)=\mathbb {L}_{}(W)\oplus\Big(\L(t_{1},\dots,t_{m},t_v,t_{\{v',u\}}\mid v,v',u\in V(\G), ;\partial\Big)$ 
	be the DGL obtained from  $\mathbb {L}_{}(W)$ by adding generators $t_{1},\dots,t_{m},t_v,t_{(v',u)}$ in degree $p+1$.  The differential is given by 
	\begin{equation}\label{xa11}
		\partial(t_{1})=y_{1},\dots,\partial(t_{m})=y_{m},\,\,\,\,\,\,\,\,
		\partial(t_{v})=y_{v},\,\,\,\,\,\,\,\,\,\,\,\,
		\partial(t_{\{v',u\}})=y_{\{v',u\}}.
	\end{equation}
	If $[\alpha]\in \mathcal{E}(\mathbb {L}(U))$ is such that
	\begin{equation}\label{xa10}
		\alpha(y_i)=y_i\,\,\,\,\,\,\,\,\alpha(y_v)=y_\phi(v),\,\,\,\,\,\,\,\,\alpha(y_{\{v',u\}})=y_{\{\phi(v'),\phi(u)\}},
	\end{equation}	
	then  we have
	\begin{eqnarray}
		\label{j0}
		\alpha_{}(t_{\{v',u\}})\hspace{-2mm}&=&\hspace{-2mm}t_{\{\phi(v'),\phi(u)\}}+C_{\{v',u\}},\,\,\,\,\,\,\,\,\,\,\,\,\,\,\,\,\,\,\,\forall v',u\in V(\G)\nonumber\\
		\alpha_{}(t_{v})\hspace{-2mm}&=&\hspace{-2mm}t_{\phi(v)}+C_{v},\,\,\,\,\,\,\,\,\,\,\,\,\,\,\,\,\,\,\,\,\,\,\,\,\,\,\,\,\,\,\,\,\,\,\,\,\,\,\,\,\,\,\,\,\,\,\,\,\forall v\in V(\G)\nonumber\\
		\alpha_{}(t_{k})\hspace{-2mm}&=&\hspace{-2mm}t_{k}+C_{k},\,\,\,\,\,\,\,\,\,\,\,\,\,\,\,\,\,\,\,\,\,\,\,\,\,\,\,\,\,\,\,\,\,\,\,\,\,\,\,\,\,\,\,\,\,\,\,\,\,\,\,\,\,\,\,\forall k=1,\dots,m.\nonumber
	\end{eqnarray}
	where $C_{\{v',u\}}$, $C_{v}$ and $C_{k}$ are cycles of degree $p+1$.
\end{lemma}
\begin{proof}
	First, for degree reasons, we can write
	\begin{eqnarray}
		\label{xa12}
		\alpha_{}(t_{\{v',u\}})\hspace{-2mm}&=&\hspace{-2mm}\sum_{i=1}^{m}\sigma_{i}t_{i}+\sum_{r\in V(\G)}\sigma_{\{v',u\},r}t_r+\sum_{s,t\in V(\G)}\sigma_{\{v',u\},\{s,t\}}t_{\{s,t\}}+C_{\{v',u\}},\nonumber\\
		\alpha_{}(t_{v})\hspace{-2mm}&=&\hspace{-2mm}\sum_{i=1}^{m}\tau_{i}t_{i}+\sum_{r\in V(\G)}\tau_{v,r}t_r+\sum_{s,t\in V(\G)}\tau_{v,\{s,t\}}t_{\{s,t\}}+C_{v},\nonumber\\
		\alpha_{}(t_{k})\hspace{-2mm}&=&\hspace{-2mm}\sum_{i=1}^{m}\nu_{i}t_{i}+\sum_{r\in V(\G)}\nu_{k,r}t_r+\sum_{s,t\in V(\G)}\nu_{k,\{s,t\}}t_{\{s,t\}}+C_{k}.
	\end{eqnarray}
	where 	 all the coefficients  belong to $\Z_{(p)}$ and where  $C_{\{v',u\}},C_{v}$ and $C_{k}$ are   decomposable elements in $\mathbb{L}_{p+1}(W)$. Therefore,  using  (\ref{xa11}) we get
	\begin{eqnarray}
		\label{xa13}
		\partial(\alpha_{}(t_{\{v',u\}}))\hspace{-2mm}&=&\hspace{-2mm}\sum_{i=1}^{m}\sigma_{i}\partial(t_{i})+\sum_{r\in V(\G)}\sigma_{\{v',u\},r}\partial(t_{r})+\sum_{s,t\in V(\G)}\sigma_{\{v',u\},\{s,t\}}\partial(t_{\{s,t\}})+\partial(C_{\{v',u\}}),\nonumber\\
			\partial(\alpha_{}(t_{v}))\hspace{-2mm}&=&\hspace{-2mm}\sum_{i=1}^{m}\tau_{i}\partial(t_{i})+\sum_{r\in V(\G)}\tau_{v,r}\partial(t_r)+\sum_{s,t\in V(\G)}\tau_{v,\{s,t\}}t_{\{s,t\}}+\partial(C_{v}),\nonumber\\
		\partial(	\alpha_{}(t_{k}))\hspace{-2mm}&=&\hspace{-2mm}\sum_{i=1}^{m}\nu_{i}\partial(t_{i})+\sum_{r\in V(\G)}\nu_{k,r}\partial(t_r)+\sum_{s,t\in V(\G)}\nu_{k,\{s,t\}}t_{\{s,t\}}+\partial(C_{k}).
	\end{eqnarray}
	Next,  using  \ref{xa10} we get
	\begin{equation*}
		\label{a5}
		\alpha( \partial(t_{\{v',u\}}))=	\alpha(y_{\{v',u\}})=y_{\{\phi(v'),\phi(u)\}},\,\,\,\,\, \alpha( \partial(t_{v}))=	\alpha(y_{v})=y_{\phi(v)},\,\,\,\,\,\alpha( \partial(t_{k}))=\alpha(y_{k})=y_{k}
	\end{equation*}
	As $\partial\circ \alpha=	\alpha\circ \partial$,  it follows that $\partial(C_{\{v',u\}})=\partial(C_{v})=\partial(C_{k})=0$ and all the coefficients in the relations  (\ref{xa13}) are zero except  $\sigma_{\{v',u\},\{\phi(v'),\phi(u)\}}=\tau_{v,v}=\nu_{k}=1$. 
\end{proof}
\subsection{Construction of the DGL $\mathcal {L}_{}(\G,1)$}
We extend 	$\mathcal {L}_{}(\G,0)$ by adding generators to  obtain  the following DGL
\begin{equation*}
	\mathcal {L}_{}(\G,1)=\mathcal {L}_{}(\G,0)\oplus\Big(\L(t_{1,1},\dots,t_{1,m_0},t_{1,v},t_{1,\{v',u\}};\mid v,v',u\in V(\G),\partial\Big).
\end{equation*}
The degrees of the generators  are as follows 
$$\vert t_{1,1}\vert=\dots=\vert t_{1,m_1}\vert=\vert t_{1,v}\vert=\vert t_{1,\{v',u\}}\vert =2314,\,\,\,\,\, \forall v,v',u\in V(\G).$$
The differential is given by
\begin{eqnarray}
	\label{xa9}
	\partial(t_{1,1})=y_{0,1},\dots,\partial(t_{1,m_1})=y_{0,m_0},\,\,\,\,\,\,\,\,
	\partial(t_{1,v})=y_{0,v},\,\,\,\,\,\,\,\,
	\partial(t_{1,\{v',u\}})=y_{0,\{v',u\}}.
\end{eqnarray}	
where $\B_0=\{y_{0,1},\dots,y_{0,m_0},y_{0,v},y_{0,\{v',u\}}; v,v',u\in V(\G)\}$ as  in  (\ref{z9}). 
\begin{lemma}
	\label{xll5}
	If $[\alpha]\in \mathcal{E}(\mathcal {L}(\G,1))$, then  there exists a unique  $\phi\in\aut(\G)$ such that
	\begin{eqnarray}
		\label{xya2}
		\alpha_{}(t_{1,\{v',u\}})\hspace{-2mm}&=&\hspace{-2mm}t_{1,\{\phi(v'),\phi(u)\}}+C_{1,\{v',u\}},\,\,\,\,\,\,\,\,\,\,\,\,\,\,\,\,\,\,\,\forall v',u\in V(\G)\nonumber\\
		\alpha_{}(t_{1,v})\hspace{-2mm}&=&\hspace{-2mm}t_{1,\phi(v)}+C_{1,v},\,\,\,\,\,\,\,\,\,\,\,\,\,\,\,\,\,\,\,\,\,\,\,\,\,\,\,\,\,\,\,\,\,\,\,\,\,\,\,\,\,\,\,\,\,\,\,\,\forall v\in V(\G)\nonumber\\
		\alpha_{}(t_{1,k})\hspace{-2mm}&=&\hspace{-2mm}t_{1,k}+C_{1,k},\,\,\,\,\,\,\,\,\,\,\,\,\,\,\,\,\,\,\,\,\,\,\,\,\,\,\,\,\,\,\,\,\,\,\,\,\,\,\,\,\,\,\,\,\,\,\,\,\,\,\,\,\,\,\,\forall k=1,\dots,m_1.\nonumber\\
		\alpha(z_{(v,u)})\hspace{-2mm}&=&\hspace{-2mm}z_{(\phi(v),\phi(u))}+B_{(v,u)},\,\,\,\,\,\,\,\,\,\,\,\,\,\,\,\,\,\,\,\,\,\,\,\,\,\,\,\,\,\,\,\,\,\forall (v,u)\in E(\G),\nonumber\\	
		\alpha_{}(x_v)\hspace{-2mm}&=&\hspace{-2mm}x_{\phi(v)},\,\,\,\,\,\,\,\,\,\,\,\,\,\,\,\,\,\,\,\,\,\,\,\,\,\,\,\,\,\,\,\,\,\,\,\,\,\,\,\,\,\,\,\,\,\,\,\,\,\,\,\,\,\,\,\,\,\,\,\,\,\,\,\,\,\,\,\,\,\,\forall v\in V(\G),\nonumber\\
		\alpha_{}(w_i)	\hspace{-2mm}&=&\hspace{-2mm}w_i,\,\,\,\,\,\,\,\,\,\,\,\,\,\,\,\,\,\,\,\,\,\,\,\,\,\,\,\,\,\,\,\,\,\,\,\,\,\,\,\,\,\,\,\,\,\,\,\,\,\,\,\,\,\,\,\,\,\,\,\,\,\,\,\,\,\,\,\,\,\,\,\,\,\,\,\,\,\,i=1,2,3,4,5,6\nonumber.
	\end{eqnarray}
where $C_{1,\{v',u\}}$, $C_{1,v}$, $C_{1,k}$ are cycles of degree $2314$.
\end{lemma}
\begin{proof}
First, upon invoking Lemma \ref{xz1} and  considering Remark \ref{m0}, we get
	\begin{eqnarray}
	\alpha_{}(t_{1,\{v',u\}})\hspace{-2mm}&=&\hspace{-2mm}t_{1,\{\phi(v'),\phi(u)\}}+C_{1,\{v',u\}},\,\,\,\,\,\,\,\,\,\,\,\,\,\,\,\,\,\,\,\forall v',u\in V(\G)\nonumber\\
	\alpha_{}(t_{1,v})\hspace{-2mm}&=&\hspace{-2mm}t_{1,\phi(v)}+C_{1,v},\,\,\,\,\,\,\,\,\,\,\,\,\,\,\,\,\,\,\,\,\,\,\,\,\,\,\,\,\,\,\,\,\,\,\,\,\,\,\,\,\,\,\,\,\,\,\,\,\forall v\in V(\G)\nonumber\\
	\alpha_{}(t_{1,k})\hspace{-2mm}&=&\hspace{-2mm}t_{1,k}+C_{1,k},\,\,\,\,\,\,\,\,\,\,\,\,\,\,\,\,\,\,\,\,\,\,\,\,\,\,\,\,\,\,\,\,\,\,\,\,\,\,\,\,\,\,\,\,\,\,\,\,\,\,\,\,\,\,\,\forall k=1,\dots,m_1.\nonumber
\end{eqnarray}
Next, in one hand, by  Proposition  \ref{p4}, we know that $\alpha(w_6)=w_6+F$, where the cycle  $F$ is a linear combination of the elements of the base  $\B_{0}$. On the other   hand,   by     (\ref{xa9}),  each element of $\B_{0}$
is a boundary. Thus, by Lemma \ref{l0} and the relation (\ref{dd1}),  the DGL-map $\alpha$ can be chosen, up to homotopy, such that   
$\alpha(w_6)=w_6$. 
\end{proof} 
\begin{lemma}
	\label{zll4} The sub-module $Z_{2314}(\mathcal{L}(\G,1))$  is not trivial. 
\end{lemma}
\begin{proof}
	It is easy to check that the two following brackets
	$$\big[w_3,[w_3,[w_2,(\operatorname{ad} w_{1})^{14}(w_2)]]\big],\,\,\,\,\,\,\,\,\big[x_v,[w_3,[w_3,[w_2,(\operatorname{ad} w_{1})^{8}(w_2)]]]\big]$$
	are  cycles of degree $2314$. 
\end{proof}
\begin{remark}
	\label{xl5}
Since $C_{1,\{v',u\}}$, $C_{1,v}$ and $C_{1,k}$ are cycles of degree $2314$, by Remark \ref{m0},  we can write each of them as a linear combination of the elements of   $\B_{1}$.
\end{remark}
Using  the preceding process, for every  $2 \leq s \leq 24$,  we construct  a   DGL denoted by 
\begin{equation*}
	\mathcal {L}_{}(\G,s)=\mathcal {L}_{}(\G,s-1)\oplus\Big(\L(t_{s,1},\dots,t_{s,m_{s-1}},t_{s,v},t_{s,\{v',u\}};\mid v,v',u\in V(\G),\partial\Big).
\end{equation*}
 that fulfill the following properties  for all $v,v',u\in V(\G)$.
\begin{enumerate}
\item	$\vert t_{s,1}\vert=\dots=\vert t_{s,m_s}\vert=\vert t_{s,v}\vert=\vert t_{s,\{v',u\}}\vert =2313+s.$
\item $\partial(t_{s,1})=y_{s-1,1},\dots,\partial(t_{s,m_{s-1}})=y_{s-1,m_{s-1}},\,\,\,\,\,
\partial(t_{s,v})=y_{s-1,v},\,\,\,\partial(t_{s,\{v',u\}})=y_{s-1,\{v',u\}}$.
\item All the cycles $y_{s-1,1},\dots,y_{s-1,m_s};y_{s-1,v};y_{s-1,\{v',u\}}$  form the base $\B_{s}$ of the sub-module $Z_{2313+s}(\mathcal{L}(\G,s))$ as it is mentioned in Remark \ref{m0}.
\item 	If $[\alpha]\in \mathcal{E}(\mathcal {L}(\G,24))$, then  there exists a unique  $\phi\in\aut(\G)$ such that
\begin{eqnarray}
	\label{dd1}
	\alpha_{}(t_{s,\{v',u\}})\hspace{-2mm}&=&\hspace{-2mm}t_{s,\{\phi(v'),\phi(u)\}},\,\,\,\,\,\,\,\,\,\,\,\,\,\,\,\,\,\,\,\,\,\,\,\,\,\,\,\,\,\,\,\,\,\,\,\,\,\,\,\,\,\,\,\,\,\,\,\,\forall v',u\in V(\G)\nonumber\\
	\alpha_{}(t_{s,v})\hspace{-2mm}&=&\hspace{-2mm}t_{s,\phi(v)},\,\,\,\,\,\,\,\,\,\,\,\,\,\,\,\,\,\,\,\,\,\,\,\,\,\,\,\,\,\,\,\,\,\,\,\,\,\,\,\,\,\,\,\,\,\,\,\,\,\,\,\,\,\,\,\,\,\,\,\,\,\,\,\,\,\,\,\forall v\in V(\G)\nonumber\\
	\alpha_{}(t_{s,k})\hspace{-2mm}&=&\hspace{-2mm}t_{s,k},\,\,\,\,\,\,\,\,\,\,\,\,\,\,\,\,\,\,\,\,\,\,\,\,\,\,\,\,\,\,\,\,\,\,\,\,\,\,\,\,\,\,\,\,\,\,\,\,\,\,\,\,\,\,\,\,\,\,\,\,\,\,\,\,\,\,\,\,\,\,\,\,\,\forall k=1,\dots,m_s.\nonumber\\
	\alpha(z_{(v,u)})\hspace{-2mm}&=&\hspace{-2mm}z_{(\phi(v),\phi(u))}+B_{(v,u)},\,\,\,\,\,\,\,\,\,\,\,\,\,\,\,\,\,\,\,\,\,\,\,\,\,\,\,\,\,\,\,\,\,\forall (v,u)\in E(\G),\\
	\alpha_{}(x_v)\hspace{-2mm}&=&\hspace{-2mm}x_{\phi(v)},\,\,\,\,\,\,\,\,\,\,\,\,\,\,\,\,\,\,\,\,\,\,\,\,\,\,\,\,\,\,\,\,\,\,\,\,\,\,\,\,\,\,\,\,\,\,\,\,\,\,\,\,\,\,\,\,\,\,\,\,\,\,\,\,\,\,\,\,\,\,\,\forall v\in V(\G),\nonumber\\
	\alpha_{}(w_i)	\hspace{-2mm}&=&\hspace{-2mm}w_i,\,\,\,\,\,\,\,\,\,\,\,\,\,\,\,\,\,\,\,\,\,\,\,\,\,\,\,\,\,\,\,\,\,\,\,\,\,\,\,\,\,\,\,\,\,\,\,\,\,\,\,\,\,\,\,\,\,\,\,\,\,\,\,\,\,\,\,\,\,\,\,\,\,\,\,\,\,\,\,i=1,2,3,4,5,6\nonumber.
\end{eqnarray}
\end{enumerate}
\begin{lemma}	
	\label{b9} The sub-module  $Z_{2337}(\mathcal{L}(\G,24))$ is not trivial.
\end{lemma}	
\begin{proof}
	Indeed, for instance,	it is easy to check that the following brackets
	$$[[x_v,x_u],[[x_2,x_4],[x_2,[x_2,w_3]]],\,\,\,\,\,\,\,\,\,\,\,\,\,\,\,\,\,\,\,\,\,\,\,\,\,\,\,\,[[w_2,w_4],(\operatorname{ad} w_{1})^{6}((\operatorname{ad} w_{2})^{3}(x_{v}))],$$
	$$[[w_3,w_5],[[w_2,w_4],[w_3,(\operatorname{ad} w_{1})^{5}(w_{3})\big)]]],$$
	are cycles of degree $2337$. 
\end{proof}
\begin{corollary}
	\label{c4}
As  $B_{(v,u)}$ in \rm(\ref{dd1}) is a cycle, according to Proposition \ref{p4},  it can be written as a linear combination of  elements in   $\B_{24}$.
\end{corollary}
\subsection{Construction of the DGL $\mathcal {L}_{}(\G,25)$}
We extend 	$\mathcal {L}_{}(\G,24)$ by adding generators to  define  the following DGL
\begin{equation*}
	\mathcal {L}_{}(\G,25)=\mathcal {L}_{}(\G,24)\oplus\Big(\L(t_{25,1},\dots,t_{25,m_{24}},t_{25,v},t_{25, \{v,u\}}\mid v,v',u\in V(\G)),\partial\Big).
\end{equation*}
The degrees of the generators  are as follows 
$$\vert t_{25,1}\vert=\dots=\vert t_{25,m_{24}}\vert=\vert t_{25,v}\vert=\vert t_{25,\{v,u\}}\vert=2338,\,\,\,\,\, \forall v,v',u\in V(\G) .$$
The differential is given by
\begin{eqnarray}
	\label{va9}
	\partial(t_{24,1})=y_{23,1},\dots,\partial(t_{24,m_{23}})=y_{23,m_{23}}
\end{eqnarray}	
$$	\partial(t_{25,v})=y_{24,v},\,\,\,\,\,\,\,\,\,\,\,\,\,\,\,\,\,\,
\partial(t_{25,\{v,u\}})=y_{24,\{v,u\}}.$$
where the cycles $y_{24,1},\dots,y_{24,m_{24}},y_{24,v},y_{24,\{v',u\}}$ form a basis of $\B_{23}$ as in  (\ref{z9}). 
\begin{lemma}
	\label{ll4}  The sub-module  $Z_{2338}(\mathcal{L}(\G,25))$ is not trivial.
\end{lemma}
\begin{proof}
Obviously, $\big[[w_3,(\operatorname{ad} w_{2})^{3}(w_4)],[[w_2,w_4],[w_3,(\operatorname{ad} w_{1})^{5}(w_2)]]\big]$
	is a 2338-cycle.
\end{proof}
\begin{lemma}
		\label{ll5}
	If $[\alpha]\in \mathcal{E}(\mathcal {L}(\G,25))$, then  there is a unique  $\phi\in\aut(\G)$ such that
 \begin{eqnarray}
		\label{a2}
		\alpha_{}(t_{25,\{v',u\}})\hspace{-2mm}&=&\hspace{-2mm}t_{25,\{\phi(v'),\phi(u)\}}+C_{25,\{v',u\}},\,\,\,\,\,\,\,\,\,\,\,\,\,\,\,\,\,\,\,\forall v',u\in V(\G)\nonumber\\
		\alpha_{}(t_{25,v})\hspace{-2mm}&=&\hspace{-2mm}t_{25,\phi(v)}+C_{25,v},\,\,\,\,\,\,\,\,\,\,\,\,\,\,\,\,\,\,\,\,\,\,\,\,\,\,\,\,\,\,\,\,\,\,\,\,\,\,\,\,\,\,\,\,\,\,\,\,\forall v\in V(\G)\nonumber\\
		\alpha_{}(t_{25,k})\hspace{-2mm}&=&\hspace{-2mm}t_{25,k}+C_{25,k},\,\,\,\,\,\,\,\,\,\,\,\,\,\,\,\,\,\,\,\,\,\,\,\,\,\,\,\,\,\,\,\,\,\,\,\,\,\,\,\,\,\,\,\,\,\,\,\,\,\,\,\,\,\,\,\forall k=1,\dots,m_{24}.\nonumber\\
		\alpha(z_{(v,u)})\hspace{-2mm}&=&\hspace{-2mm}z_{(\phi(v),\phi(u))},\,\,\,\,\,\,\,\,\,\,\,\,\,\,\,\,\,\,\,\,\,\,\,\,\,\,\,\,\,\,\,\,\,\,\,\,\,\,\,\,\,\,\,\,\,\,\,\,\,\,\,\,\,\,\,\,\,\,\,\,\,\forall (v,u)\in E(\G),\nonumber\\	
		\alpha_{}(x_v)\hspace{-2mm}&=&\hspace{-2mm}x_{\phi(v)},\,\,\,\,\,\,\,\,\,\,\,\,\,\,\,\,\,\,\,\,\,\,\,\,\,\,\,\,\,\,\,\,\,\,\,\,\,\,\,\,\,\,\,\,\,\,\,\,\,\,\,\,\,\,\,\,\,\,\,\,\,\,\,\,\,\,\,\,\,\,\,\,\,\,\,\forall v\in V(\G),\nonumber\\
		\alpha_{}(w_i)	\hspace{-2mm}&=&\hspace{-2mm}w_i,\,\,\,\,\,\,\,\,\,\,\,\,\,\,\,\,\,\,\,\,\,\,\,\,\,\,\,\,\,\,\,\,\,\,\,\,\,\,\,\,\,\,\,\,\,\,\,\,\,\,\,\,\,\,\,\,\,\,\,\,\,\,\,\,\,\,\,\,\,\,\,\,\,\,\,\,\,\,\,\,\,\,i=1,2,3,4,5,6\nonumber.
	\end{eqnarray}
where $C_{25,\{v',u\}}$, $C_{25,v}$, $C_{25,k}$ are $2338$-cycles. 
\end{lemma}
\begin{proof}
First, if we apply Lemma \ref{xz1}, then we get
\begin{eqnarray}
	\alpha_{}(t_{24,\{v',u\}})\hspace{-2mm}&=&\hspace{-2mm}t_{24,\{\phi(v'),\phi(u)\}}+C_{24,\{v',u\}},\,\,\,\,\,\,\,\,\,\,\,\,\,\,\,\,\,\,\,\forall v',u\in V(\G)\nonumber\\
	\alpha_{}(t_{24,v})\hspace{-2mm}&=&\hspace{-2mm}t_{24,\phi(v)}+C_{24,v},\,\,\,\,\,\,\,\,\,\,\,\,\,\,\,\,\,\,\,\,\,\,\,\,\,\,\,\,\,\,\,\,\,\,\,\,\,\,\,\,\,\,\,\,\,\,\,\,\forall v\in V(\G)\nonumber\\
	\alpha_{}(t_{24,k})\hspace{-2mm}&=&\hspace{-2mm}t_{24,k}+C_{24,k},\,\,\,\,\,\,\,\,\,\,\,\,\,\,\,\,\,\,\,\,\,\,\,\,\,\,\,\,\,\,\,\,\,\,\,\,\,\,\,\,\,\,\,\,\,\,\,\,\,\,\,\,\,\,\,\forall k=1,\dots,m_{23}.\nonumber
\end{eqnarray}
Next,  by  Corollary \ref{c4}, we know that  $B_{(v,u)}$ is the linear combination of the elements of  $\B_{24}$. But,    by     (\ref{va9}),    each element of $\B_{24}$
is a boundary. Thus, by Lemma \ref{l0} and  (\ref{dd1}),  the DGL-map $\alpha$ can be chosen, up to homotopy, such that   
$\alpha(z_{(v,u)})=z_{(\phi(v),\phi(u))}$. 
\end{proof}
\begin{remark}
\label{xlx5}
Since $C_{24,\{v,u\}}$, $C_{24,v}$ and $C_{24,k}$ are $2338$-cycles, by Remark \ref{m0} we can write each of them as a linear combination of the elements of   $\B_{25}$.
\end{remark}
\subsection{Construction of the DGL 	$\mathcal {L}_{}(\G,26)$}
We extend 	$\mathcal {L}_{}(\G,25)$ by adding generators to  define  the following DGL
\begin{equation*}
	\mathcal {L}_{}(\G,26)=\mathcal {L}_{}(\G,25)\oplus\Big(\L(t_{26,1},\dots,t_{26,m_{25}},t_{26,v},t_{26, \{v,u\}}\mid v,v',u\in V(\G)),\partial\Big).
\end{equation*}
The degrees of the generators  are as follows 
$$\vert t_{26,1}\vert=\dots=\vert t_{26,m_{25}}\vert=\vert t_{26,v}\vert=\vert t_{26,\{v,u\}}\vert=2339,\,\,\,\,\, \forall v,v',u\in V(\G) .$$
The differential is given by
\begin{eqnarray}
	\label{xva9}
	\partial(t_{26,1})=y_{25,1},\dots,\partial(t_{26,m_{25}})=y_{25,m_{25}}
\end{eqnarray}	
$$	\partial(t_{26,v})=y_{25,v},\,\,\,\,\,\,\,\,\,\,\,\,\,\,\,\,\,\,
\partial(t_{26,\{v,u\}})=y_{25,\{v,u\}}.$$
where the cycles $y_{25,1},\dots,y_{25,m_{25}},y_{25,v},y_{25,\{v',u\}}$ form a basis of $\B_{25}$ as in  (\ref{z9}).
\begin{lemma}
	\label{zll6} The sub-module  $Z_{2339}(\mathcal{L}(\G,26))$ is not trivial. 
\end{lemma}
\begin{proof}
Obviously, $\big[[w_3,[w_2,w_4]],[[w_2,w_4],[[w_2,w_4],(\operatorname{ad} w_{1})^{5}(w_3)]]\big],$
	is a $2339$-cycle.
	\end{proof}
\begin{lemma}
	\label{ll8}
Let $[\alpha]\in \mathcal{E}(\mathcal {L}(\G,26))$. There is a unique  $\phi\in\aut(\G)$ such that
	\begin{eqnarray}
		\label{a18}
		\alpha_{}(t_{26,\{v',u\}})\hspace{-2mm}&=&\hspace{-2mm}t_{26,\{\phi(v'),\phi(u)\}}+C_{26,\{v',u\}},\,\,\,\,\,\,\,\,\,\,\,\,\,\,\,\,\,\,\forall v',u\in V(\G)\nonumber\\
		\alpha_{}(t_{26,v})\hspace{-2mm}&=&\hspace{-2mm}t_{26,\phi(v)}+C_{26,v},\,\,\,\,\,\,\,\,\,\,\,\,\,\,\,\,\,\,\,\,\,\,\,\,\,\,\,\,\,\,\,\,\,\,\,\,\,\,\,\,\,\,\,\,\,\,\forall v\in V(\G)\nonumber\\
		\alpha_{}(t_{26,l})\hspace{-2mm}&=&\hspace{-2mm}t_{26,l}+C_{26,l},\,\,\,\,\,\,\,\,\,\,\,\,\,\,\,\,\,\,\,\,\,\,\,\,\,\,\,\,\,\,\,\,\,\,\,\,\,\,\,\,\,\,\,\,\,\,\,\,\,\,\,\,\,\,\,\forall l=1,\dots,m_{25}.\nonumber\\
		\alpha_{}(t_{25,\{v',u\}})\hspace{-2mm}&=&\hspace{-2mm}t_{25,\{\phi(v'),\phi(u)\}},\,\,\,\,\,\,\,\,\,\,\,\,\,\,\,\,\,\,\,\,\,\,\,\,\,\,\,\,\,\,\,\,\,\,\,\,\,\,\,\,\,\,\,\,\,\,\,\,\,\,\forall v',u\in V(\G)\nonumber\\
		\alpha_{}(t_{25,v})\hspace{-2mm}&=&\hspace{-2mm}t_{25,\phi(v)},\,\,\,\,\,\,\,\,\,\,\,\,\,\,\,\,\,\,\,\,\,\,\,\,\,\,\,\,\,\,\,\,\,\,\,\,\,\,\,\,\,\,\,\,\,\,\,\,\,\,\,\,\,\,\,\,\,\,\,\,\,\,\,\,\,\,\,\forall v\in V(\G)\nonumber\\
		\alpha_{}(t_{25,k})\hspace{-2mm}&=&\hspace{-2mm}t_{25,k},\,\,\,\,\,\,\,\,\,\,\,\,\,\,\,\,\,\,\,\,\,\,\,\,\,\,\,\,\,\,\,\,\,\,\,\,\,\,\,\,\,\,\,\,\,\,\,\,\,\,\,\,\,\,\,\,\,\,\,\,\,\,\,\,\,\,\,\,\,\,\,\,\,\,\forall k=1,\dots,m_{24}.\nonumber
	\end{eqnarray}
where $C_{26,\{v',u\}}$, $C_{26,v}$, $C_{26,l}$ are $2339$-cycles.
\end{lemma}
\begin{proof}
First, if we apply Lemma \ref{xz1}, then we get
\begin{eqnarray}
	\alpha_{}(t_{26,\{v',u\}})\hspace{-2mm}&=&\hspace{-2mm}t_{26,\{\phi(v'),\phi(u)\}}+C_{26,\{v',u\}},\,\,\,\,\,\,\,\,\,\,\,\,\,\,\,\,\,\,\,\forall v',u\in V(\G)\nonumber\\
	\alpha_{}(t_{26,v})\hspace{-2mm}&=&\hspace{-2mm}t_{26,\phi(v)}+C_{26,v},\,\,\,\,\,\,\,\,\,\,\,\,\,\,\,\,\,\,\,\,\,\,\,\,\,\,\,\,\,\,\,\,\,\,\,\,\,\,\,\,\,\,\,\,\,\,\,\,\forall v\in V(\G)\nonumber\\
	\alpha_{}(t_{26,l})\hspace{-2mm}&=&\hspace{-2mm}t_{26,l}+C_{26,l},\,\,\,\,\,\,\,\,\,\,\,\,\,\,\,\,\,\,\,\,\,\,\,\,\,\,\,\,\,\,\,\,\,\,\,\,\,\,\,\,\,\,\,\,\,\,\,\,\,\,\,\,\,\,\,\,\,\forall l=1,\dots,m_{25}.\nonumber
\end{eqnarray}
Next,  from Lemma \ref{ll5}, we know that  $C_{25,\{v,u\}}$, $C_{25,v}$ and $C_{25,k}$ are cycles which can be written as   linear combinations of the elements in  $\B_{25}$ due to Remark \ref{m0}. But ,   by     (\ref{xva9}),    each element of $\B_{25}$ is a boundary. Thus, by Lemma \ref{l0},  the DGL-map $\alpha$ can be chosen, up to homotopy, such that   for every $v,v',u\in V(\G)$ and  $1\leq k\leq m_{24}$ we have
$$\alpha_{}(t_{24,\{v',u\}})=t_{24,\{\phi(v'),\phi(u)\}},\,\,\,
\alpha_{}(t_{24,v})=t_{24,\phi(v)},\,\,\,\,\,
\alpha_{}(t_{24,k})=t_{24,k}.$$ 
as desired.	
\end{proof}
\subsection{Construction of  $\mathcal {L}_{}(\G,27)$}
Define 
\begin{equation*}
\mathcal {L}_{}(\G,27)=\mathcal {L}_{}(\G,26)\oplus\Big(\L(t_{27,1},\dots,t_{27,m_{26}},t_{26,v},t_{27, \{v,u\}}\mid v,v',u\in V(\G)),\partial\Big).
\end{equation*}
The degrees of the generators  are as follows 
$$\vert t_{27,1}\vert=\dots=\vert t_{27,m_{24}}\vert=\vert t_{27,v}\vert=\vert t_{27,\{v,u\}}\vert=2340,\,\,\,\,\, \forall v,v',u\in V(\G) .$$
The differential is given by
\begin{eqnarray}
\label{m2}
\partial(t_{27,1})=y_{26,1},\dots,\partial(t_{27,m_{26}})=y_{27,m_{26}}
\end{eqnarray}	
$$	\partial(t_{27,v})=y_{26,v},\,\,\,\,\,\,\,\,\,\,\,\,\,\,\,\,\,\,
\partial(t_{27,\{v',u\}})=y_{26,\{v',u\}}.$$
where the cycles $y_{26,1},\dots,y_{26,m_{26}},y_{26,v},y_{26,\{v',u\}}$ form a basis of $\B_{26}$ as in  (\ref{z9}).
\begin{lemma}
	\label{xll8}
	Let $[\alpha]\in \mathcal{E}(\mathcal {L}(\G,27))$. There is a unique  $\phi\in\aut(\G)$ such that
	\begin{eqnarray}
		\label{xa18}
		\alpha_{}(t_{27,\{v',u\}})\hspace{-2mm}&=&\hspace{-2mm}t_{27,\{\phi(v'),\phi(u)\}}+C_{27,\{v',u\}},\,\,\,\,\,\,\,\,\,\,\,\,\,\,\,\,\,\,\,\,\forall v',u\in V(\G)\nonumber\\
		\alpha_{}(t_{27,v})\hspace{-2mm}&=&\hspace{-2mm}t_{27,\phi(v)}+C_{27,v},\,\,\,\,\,\,\,\,\,\,\,\,\,\,\,\,\,\,\,\,\,\,\,\,\,\,\,\,\,\,\,\,\,\,\,\,\,\,\,\,\,\,\,\,\,\,\,\,\,\forall v\in V(\G)\nonumber\\
		\alpha_{}(t_{27,l})\hspace{-2mm}&=&\hspace{-2mm}t_{27,l}+C_{27,l},\,\,\,\,\,\,\,\,\,\,\,\,\,\,\,\,\,\,\,\,\,\,\,\,\,\,\,\,\,\,\,\,\,\,\,\,\,\,\,\,\,\,\,\,\,\,\,\,\,\,\,\,\,\,\,\,\,\,\forall l=1,\dots,m_{26}.\nonumber\\
		\alpha_{}(t_{26,\{v',u\}})\hspace{-2mm}&=&\hspace{-2mm}t_{26,\{\phi(v'),\phi(u)\}},\,\,\,\,\,\,\,\,\,\,\,\,\,\,\,\,\,\,\,\,\,\,\,\,\,\,\,\,\,\,\,\,\,\,\,\,\,\,\,\,\,\,\,\,\,\,\,\,\,\,\,\,\,\forall v',u\in V(\G)\nonumber\\
		\alpha_{}(t_{26,v})\hspace{-2mm}&=&\hspace{-2mm}t_{26,\phi(v)},\,\,\,\,\,\,\,\,\,\,\,\,\,\,\,\,\,\,\,\,\,\,\,\,\,\,\,\,\,\,\,\,\,\,\,\,\,\,\,\,\,\,\,\,\,\,\,\,\,\,\,\,\,\,\,\,\,\,\,\,\,\,\,\,\,\,\,\,\,\,\forall v\in V(\G)\nonumber\\
		\alpha_{}(t_{26,k})\hspace{-2mm}&=&\hspace{-2mm}t_{26,k},\,\,\,\,\,\,\,\,\,\,\,\,\,\,\,\,\,\,\,\,\,\,\,\,\,\,\,\,\,\,\,\,\,\,\,\,\,\,\,\,\,\,\,\,\,\,\,\,\,\,\,\,\,\,\,\,\,\,\,\,\,\,\,\,\,\,\,\,\,\,\,\,\,\,\,\,\forall k=1,\dots,m_{25}.\nonumber
	\end{eqnarray}
where $C_{27,\{v',u\}}$, $C_{27,v}$, $C_{27,k}$ are  $2340$-cycles.
\end{lemma}
\begin{proof}
First, if we apply Lemma \ref{xz1}, then we get
\begin{eqnarray}
	\alpha_{}(t_{27,\{v',u\}})\hspace{-2mm}&=&\hspace{-2mm}t_{27,\{\phi(v'),\phi(u)\}}+C_{27,\{v',u\}},\,\,\,\,\,\,\,\,\,\,\,\,\,\,\,\,\,\,\,\forall v',u\in V(\G)\nonumber\\
	\alpha_{}(t_{27,v})\hspace{-2mm}&=&\hspace{-2mm}t_{27,\phi(v)}+C_{27,v},\,\,\,\,\,\,\,\,\,\,\,\,\,\,\,\,\,\,\,\,\,\,\,\,\,\,\,\,\,\,\,\,\,\,\,\,\,\,\,\,\,\,\,\,\,\,\,\,\forall v\in V(\G)\nonumber\\
	\alpha_{}(t_{27,l})\hspace{-2mm}&=&\hspace{-2mm}t_{27,l}+C_{27,l},\,\,\,\,\,\,\,\,\,\,\,\,\,\,\,\,\,\,\,\,\,\,\,\,\,\,\,\,\,\,\,\,\,\,\,\,\,\,\,\,\,\,\,\,\,\,\,\,\,\,\,\,\,\,\,\,\,\forall l=1,\dots,m_{26}.\nonumber
\end{eqnarray}
Next,  from Lemma \ref{ll8}, we know that  $C_{26,\{v,u\}}$, $C_{26,v}$ and $C_{26,k}$ are cycles which can be written as   linear combinations of the elements in  $\B_{26}$ due to Remark \ref{m0}. But,   by     (\ref{m2}),    each element of $\B_{26}$ is a boundary. Thus by Lemma \ref{l0},  the DGL-map $\alpha$ can be chosen, up to homotopy, such that   for every $v,v',u\in V(\G)$ and  $1\leq k\leq m_{25}$ we have
$$\alpha_{}(t_{25,\{v',u\}})=t_{25,\{\phi(v'),\phi(u)\}},\,\,\,
\alpha_{}(t_{25,v})=t_{25,\phi(v)},\,\,\,\,\,
\alpha_{}(t_{25,k})=t_{25,k}.$$ 
as wanted.	
\end{proof}

\medskip
The goal of this paragraph is to show that  $Z_{2340}(\mathcal{L}(\G,26))$  is trivial. Indeed, let $$\B(w^{(k_1)}_1,w^{(k_2)}_2,w^{(k_5)}_3,w^{(k_4)}_4,w^{(k_5)}_5,x^{(k_v)}_v,x^{(k_u)}_u)\,\,\,\,\,\,\,\,\,\,\,\,\,\,\,\,\,\,\,\,\,,\,\,\,\,\,\,\,\,\,\,\,\,\,\,\,\, v\neq u,$$ denote the set of all the brackets of the Hall basis of the DGL $	\mathcal {L}_{}(\G,1)$
 formed exactly using $k_1$ generators $w_1$, $k_2$ generators $w_2$, $k_3$ generators $w_3$, $k_4$ generators $w_4$, $k_5$ generators $w_5$,  $k_v$ generators $x_v$ and  $k_u$ generators $x_u$. 
 \noindent  For instance,  the bracket
 $$[[x_v,x_u],[[x_2,x_4],[x_2,[x_2,w_3]]]\in \B(w^{(0)}_1,w^{(3)}_2,w^{(1)}_3,w^{(1)}_4,w^{(0)}_5,x^{(1)}_v,x^{(1)}_u).$$
 \begin{lemma}
	\label{ll1}
We have the following two statements. 
\begin{enumerate}
	\item Any non-zero linear combination of elements of the set	$$\B(w^{(k_1)}_1,w^{(1)}_2,w^{(k_5)}_3,w^{(2)}_4,w^{(0)}_5,x^{(k_v)}_v,x^{(k_u)}_u),$$ cannot be a cycle.
	\item Any non-zero linear combination of elements of the set	$$\B(w^{(5)}_1,w^{(3)}_2,w^{(0)}_3,w^{(3)}_4,w^{(1)}_5,x^{(0)}_v,x^{(0)}_u),$$ cannot be a cycle.
\end{enumerate}
\end{lemma}
\begin{proof}
First for the assertion (1), 	by  expanding an arbitrary non-zero element  $$l\in\B(w^{(k_1)}_1,w^{(1)}_2,w^{(k_5)}_3,w^{(2)}_4,w^{(0)}_5,x^{(k_v)}_v,x^{(k_u)}_u),$$ 
in  $\T(w_1,w_2,w_3,w_4,\{x_v\}_{ v\in V(\G)})$, we  can  write $l$ as the following  linear combination 
	$$l=c_1w^{2}_4A_1+c_2w_4A_2w_4+c_3A_3w_4A_4w_4+c_4A_5w_4A_6w_4A_7+c_5A_8w^{2}_4+c_6w_4A_9w_4A_{10},$$
where $c_1,c_2,c_3,c_4,c_5,c_6\in\Z_{(p)}$ and 	where $A_1,\dots,A_{10}$  are (non constant) expressions belonging  to  $\T(w_1,w_2,w_3,\{x_v\}_{ v\in V(\G)})$. In other words,   $A_1,\dots,A_{10}$  are  expressions of monomials composed of the generators $w_1,w_2,w_3$ and $x_v$, where $ v\in V(\G)$. Therefore, using the relations (\ref{a45})  we get
	\begin{eqnarray}
		\label{a3}
		\partial(c_1w^{2}_4A_1)\hspace{-2mm}&=&\hspace{-2mm}2c_1w^{2}_2w_4A_1\pm 2c_1w_4w^{2}_2A_1,\nonumber\\	
		\partial(c_2w_4A_2w_4)\hspace{-2mm}&=&\hspace{-2mm}2c_2w^{2}_2A_2w_4\pm 2c_2w_4A_2w^{2}_2,\nonumber\\
		\partial(c_3A_3w_4A_4w_4)\hspace{-2mm}&=&\hspace{-2mm}2c_3A_3w^{2}_2A_4w_4\pm 2c_3A_3w_4A_4w^{2}_2,\nonumber\\
		\partial(c_4A_5w_4A_6w_4A_7)\hspace{-2mm}&=&\hspace{-2mm}2c_4A_5w^{2}_2A_6w_4A_7\pm 2c_4A_5w_4A_6w^{2}_2A_7,\nonumber\\
		\partial(c_5A_8w^{2}_4)\hspace{-2mm}&=&\hspace{-2mm}2c_5A_8w^{2}_2w_4\pm 2c_5A_8w_4w^{2}_2,\nonumber\\	
		\partial(c_6w_4A_9w_4A_{10})\hspace{-2mm}&=&\hspace{-2mm}2c_6w^{2}_2A_9w_4A_{10}\pm 2c_6w_4A_9w^{2}_2A_{10}.
	\end{eqnarray}
Clearly, as the generator $w_4$ is not involved in the expressions of $A_1,\dots,A_{10}$, it follows that  $ 2c_1w_4w^{2}_2A_1$ does not appear in other expressions in (\ref{a3}). This implies that $c_1=0$. Likewise,  the expression $c_2w_4A_2w^{2}_2$ does not appear in the other expressions in (\ref{a3}). So $c_2=0$. Repeating the same argument, we conclude that   if $\partial(l)=0$, then $c_1=c_2=c_3=c_4=c_5=c_6=0$.	

\medskip
Next, for the assertion (2), by using the same argument	as in the first assertion	and   expanding an arbitrary non-zero element  $$l'\in\B(w^{(5)}_1,w^{(3)}_2,w^{(0)}_3,w^{(3)}_4,w^{(1)}_5,x^{(0)}_v,x^{(0)}_u),$$ 
in  $\T(w_1,w_2,w_3,w_4,w_5,\{x_v\}_{ v\in V(\G)})$ in this case, we  can  write $l'$ as a linear combination
	$$l'=b_1w^{2}_4B_1+b_2w_4B_2w_4+b_3B_3w_4B_4w_4+b_4B_5w_4B_6w_4B_7+b_5B_8w^{2}_4+b_6w_4B_9w_4B_{10},$$
	where $b_1,b_2,b_3,b_4,b_5,b_6\in\Q$ and 	where $B_1,\dots,B_{10}$    are (non constant)  expressions of monomials composed of the generators $w_1,w_2,w_3,w_5$ and $x_v$, where $ v\in V(\G)$. 
Therefore,  using  (\ref{a45}) we obtain
\begin{eqnarray}
	\label{a33}
	\partial(b_1w^{2}_4B_1)\hspace{-2mm}&=&\hspace{-2mm}2b_1w^{2}_2w_4B_1\pm 2b_1w_4w^{2}_2B_1\pm b_1w_4w^{2}_2\partial(B_1),\\	
	\partial(b_2w_4B_2w_4)\hspace{-2mm}&=&\hspace{-2mm}2b_2w^{2}_2B_2w_4\pm	b_2w_4\partial(B_2)w_4\pm 2b_2w_4B_2w^{2}_2,\nonumber\\
	\partial(b_3B_3w_4B_4w_4)\hspace{-2mm}&=&\hspace{-2mm}b_3\partial(B_3)w_4B_4w_4\pm 2b_3B_3w_2^{2}B_4w_4\pm b_3B_3w_4\partial(B_4)w_4\pm\nonumber\\&& 2b_3B_3w_4B_4w^{2}_2,\nonumber\\
	\partial(b_4B_5w_4B_6w_4B_7)\hspace{-2mm}&=&\hspace{-2mm}b_4\partial(B_5)w_4B_6w_4B_7\pm 2b_4B_5w_2^2B_6w_4B_7
	\pm b_4B_5w_4\partial(B_6)w_4B_7\pm \nonumber\\&&2b_4B_5w_4B_6w_2^2B_7\pm b_4B_5w_4B_6w_4\partial(B_7),\nonumber\\
	\partial(b_5B_8w^{2}_4)\hspace{-2mm}&=&\hspace{-2mm}b_5\partial(B_8)w^{2}_4\pm 2b_5B_8w^{2}_2w_4\pm 2b_5B_8w_4w^{2}_2,\nonumber\\	
	\partial(b_6w_4B_9w_4B_{10})\hspace{-2mm}&=&\hspace{-2mm}2b_6w_2^2B_9w_4B_{10}\pm b_6w_4\partial(B_9)w_4B_{10}\pm 2b_6w_4B_9w_2^2B_{10}\pm\nonumber\\
	&& b_6w_4B_9w_4\partial(B_{10})\nonumber.
\end{eqnarray}
Likewise, as the generator $w_4$ is not involved in the expressions of $B_1,\dots,B_{10}$ and $\partial(B_1),\dots,\partial(B_{10})$, it follows that  $2b_1w_4w^{2}_2B_1$ does not appear in the other expressions in (\ref{a33}). This implies that $b_1=0$. Next,  the expression $2b_2w_4B_2w^{2}_2$ does not appear in the other expressions in (\ref{a33}). So $b_2=0$. Repeating the same argument, we conclude that   if $\partial(l')=0$, then $b_3=b_4=b_5=b_6=0$.
\end{proof}	
\begin{lemma}
	\label{ll77} The module $\mathcal{L}_{2340}(\G,27)$ is  spanned by   brackets belonging to  the following two sets
	\begin{equation}\label{bb8}
		\B(w^{(5)}_1,w^{(1)}_2,w^{(2)}_3,w^{(4)}_4,w^{(0)}_5,x^{(0)}_v,x^{(0)}_u)\,\,,\,\,\B(w^{(5)}_1,w^{(3)}_2,w^{(0)}_3,w^{(3)}_4,w^{(1)}_5,x^{(0)}_v,x^{(0)}_u).
	\end{equation}
\end{lemma}
\begin{proof}
	First, recall that, by construction,  the DGL $\mathcal {L}(\G,26)$ is formed by elements of degrees
	$$115,151,201,303,403,690,q, \,\,\,\,\,\,\,\,\,\,\,2313\leq q\leq 2340.$$
	Therefore, for degree reasons, there is no bracket in $\mathcal{L}_{2340}(\G,26)$ formed with  a generator of degree $q$. Next, if  $\Theta\in\mathcal {L}_{2340}(\G,3)$,	then  we  get
	$$\vert \Theta\vert=	115a_1 +151a_2+201a_3+303a_4+403a_5+690a_6,$$
	where $a_1,\dots,a_6\in \{0,1,2,\dots\}$ yielding  the following  Frobenius equation 
	\begin{equation*}\label{bb6}
		115a_1 +151a_2+201a_3+303a_4+403a_5+690a_6=2340.
	\end{equation*}
	Using WOLFRAM software,  we get only two  solutions  which are
	providing  only two  solutions  which are
	$$a_1=5,\, \,\,a_2=1\,\,, \,\, a_3=2,\, \,\, a_4=4,\, \,\, a_5=0,\, \,\, a_6=0,$$
	$$a_1=5,\, \,\, a_2=3,\, \,\, a_3=0,\, \,\, a_4=3,\, \,\, a_5=1,\, \,\, a_6=0,$$
	as desired.
\end{proof}
\begin{lemma}
	\label{ll7} The submodule $Z_{2340}(\mathcal{L}(\G,27))$  is trivial.
\end{lemma}
\begin{proof}
	First, 	according to Lemma \ref{ll77},   $\mathcal{L}_{2340}(\G,26)$ is spanned by   brackets belonging to  the  sets (\ref{bb8}). 
	Next, by Lemma \ref{ll1}, any linear combination of elements of the set
	$B(w^{(5)}_1,w^{(1)}_2,w^{(2)}_3,w^{(4)}_4,w^{(0)}_5,x^{(0)}_v,x^{(0)}_u)$ or  $B(w^{(5)}_1,w^{(3)}_2,w^{(0)}_3,w^{(3)}_4,w^{(1)}_5,x^{(0)}_v,x^{(0)}_u)$ cannot be a cycle.
Finally, let
$$\theta=(\gamma_1l_{1}+\dots+\gamma_{k_1}l_{k_1})+(\mu_1h_{1}+\dots+\mu_{k_2}h_{k_2}),\,\,\,\,\,\,\,\,\,\,\,\gamma_1,\dots,\gamma_{k_1};\mu_1,\dots,\mu_{k_2}\in\Z_{(p)},$$ 
a linear combination such that  
$$l_{1},\dots,l_{k_1}\in\B(w^{(5)}_1,w^{(1)}_2,w^{(2)}_3,w^{(4)}_4,w^{(0)}_5,x^{(0)}_v,x^{(0)}_u),$$
$$h_{1},\dots,h_{k_2}\in\B(w^{(5)}_1,w^{(3)}_2,w^{(0)}_3,w^{(3)}_4,w^{(1)}_5,x^{(0)}_v,x^{(0)}_u).$$
Applying the differential $\partial$ we get
	\begin{equation*}
		\label{a24}
		\partial(\theta)=\gamma_1\partial(l_1) +\dots +\gamma_{k_1}\partial(l_{k_1})+\mu_1\partial(h_1) +\dots +\mu_{k_2}\partial(h_{k_2}).
	\end{equation*}
Let's write $\partial(h_j)=h_{j,1}+h_{j,2},$ where
$$\partial(l_1),\dots,\partial(l_{k_1})\in \B(w^{(5)}_1,w^{(3)}_2,w^{(2)}_3,w^{(3)}_4,w^{(0)}_5,x^{(0)}_v,x^{(0)}_u),$$
$$h_{1,1},\dots,h_{k_2,1}\in \B(w^{(5)}_1,w^{(5)}_2,w^{(0)}_3,w^{(2)}_4,w^{(1)}_5,x^{(0)}_v,x^{(0)}_u),$$
$$h_{1,2},\dots,h_{k_2,2}\in \B(w^{(5)}_1,w^{(5)}_2,w^{(2)}_3,w^{(2)}_4,w^{(0)}_5,x^{(0)}_v,x^{(0)}_u).$$
All the  brackets $h_{1,1},\dots,h_{k_2,1}$ are formed using only one  generator $w_{5}$,  while in the  other elements $w_{5}$   is not involved, so if $\theta$ is a cycle, then $\mu_1=\dots=\mu_{k_2}=0$ implying also that $\gamma_1=\dots=\gamma_{k_1}=0$. Hence,  the only $2340$-cycle  in $\mathcal{L}(\G,27)$ is zero.
\end{proof}
\begin{corollary}
	\label{xll9}
The cycles $C_{27,\{v',u\}}$, $C_{27,v}$, $C_{27,k}$, given in Lemma \rm\ref{xll8} are  trivial.
\end{corollary}
\begin{proof}
	It suffices to apply Lemma \ref{ll7}.
\end{proof}
As a consequence of Remark \ref{xl5}, Lemmas  \ref{ll5}, \ref{ll8},  \ref{xll8} and Corollary \ref{xll9},  we can define   a  group  homomorphism $\mathcal{E}(\mathcal {L}(\G,27))\to \aut(\G)$ by setting $\Psi([\alpha_{}])=\phi.$
\begin{theorem}\label{t2}
	The  map $\Psi$ is an isomorphism of groups.
\end{theorem}
\begin{proof}
	For every $\sigma\in \aut(\G)$, we define
	\begin{eqnarray}
		\label{181}
		\alpha_{\sigma}(t_{s,k})\hspace{-2mm}&=&\hspace{-2mm}t_{s,k},\,\,\,\,\,\,\,\,\,\,\,\,\,\,\,\,\,\,\,\,\,\,\,\,\,\,\,\,\,\,\,\,\,\,\,\,\,\,\,\,\,\,\,\,\,\,\,\,\,\,\,\,\,\,\,\forall k=1,\dots,m_{s-1},\,\,\,\,\,\,\,\,\,\forall s=1,\dots,27,\nonumber\\
		\alpha_{\sigma}(t_{s,\{v,u\}})\hspace{-2mm}&=&\hspace{-2mm}t_{s,\{\sigma(v),\sigma(u)\}},\,\,\,\,\,\,\,\,\,\,\,\,\,\,\,\,\,\,\,\,\,\,\,\,\,\,\,\,\,\,\,\,\,\forall v,u\in V(\G),\,\,\,\,\,\,\,\,\,\,\,\,\,\,\,\,\,\,\,\,\forall s=1,\dots,27,\nonumber\\
		\alpha_{\sigma}(t_{s,v})\hspace{-2mm}&=&\hspace{-2mm}t_{s,\sigma(v)},\,\,\,\,\,\,\,\,\,\,\,\,\,\,\,\,\,\,\,\,\,\,\,\,\,\,\,\,\,\,\,\,\,\,\,\,\,\,\,\,\,\,\,\,\,\,\,\,\,\forall v\in V(\G),,\,\,\,\,\,\,\,\,\,\,\,\,\,\,\,\,\,\,\,\,\,\,\,\forall s=1,\dots,27,\nonumber\\	
		\alpha_{\sigma}(z_{(v,u)})\hspace{-2mm}&=&\hspace{-2mm}z_{(\sigma(v),\sigma(u))},\,\,\,\,\,\,\,\,\,\,\,\,\,\,\,\,\,\,\,\,\,\,\,\,\,\,\,\,\,\,\,\,\,\,\,\,\,\,\forall (v,u)\in E(\G),\nonumber\\
		\alpha_{\sigma}(x_v)\hspace{-2mm}&=&\hspace{-2mm}x_{\sigma(v)},\,\,\,\,\,\,\,\,\,\,\,\,\,\,\,\,\,\,\,\,\,\,\,\,\,\,\,\,\,\,\,\,\,\,\,\,\,\,\,\,\,\,\,\,\,\,\,\,\,\,\,\,\forall v\in V(\G),\nonumber\\
		\alpha_{\sigma}(w_i)	\hspace{-2mm}&=&\hspace{-2mm}w_i,\,\,\,\,\,\,\,\,\,\,\,\,\,\,\,\,\,\,\,\,\,\,\,\,\,\,\,\,\,\,\,\,\,\,\,\,\,\,\,\,\,\,\,\,\,\,\,\,\,\,\,\,\,\,\,\,\,\,\,\,i=1,2,3,4,5,6\nonumber.
	\end{eqnarray}
	Clearly, we have $\partial\circ \alpha_{\sigma}=\alpha_{\sigma}\circ \partial$ implying that   $[\alpha_{\sigma}]\in \mathcal{E}(\mathcal {L}(\G,27))$. Hence, we get a map $$\Phi:\aut(\G)\to \mathcal{E}(\mathcal {L}(\G,27)),\,\,\,\,\,\,\,\,\,\,\,\,\,\,\,\,\,\,\,\,\,\,\,\,\,\, \Phi(\sigma)=[\alpha_{\sigma}],$$  and it is easy to check that it is  the inverse  of $\Psi$. Finally,  $\Phi$ is a homomorphism of groups because we have
	$\Phi(\sigma_1\circ\sigma_2)=[\alpha_{\sigma_1\circ\sigma_2}]=[\alpha_{\sigma_1}]\circ[\alpha_{\sigma_2}]=\Phi(\sigma_1)\circ\Phi(\sigma_2),$
	for all   $\sigma_1,\sigma_2\in \aut(\G)$.
\end{proof}
As a consequence of Theorem \ref{t2}, we derive the following   result.
\begin{theorem}
	\label{01}
For any group $G$ and  any prime  $p>1114$, there exists a  CW-complex $X$ such that  $G\cong\E(X_{(p)})$, where $X_{(p)}$ is the $p$-localization of $X$. More precisely:  
\begin{enumerate}
	\item $X$ is an $116$-connected, $2341$-dimensional, and of finite type if  $G$ is finite.
	\item $H_{i}(X,\Z_{(p)})$ is a free $\Z_{(p)}$-modules over a  basis which in bijection with $G$ for 
	$$i=691,q, \,\,\,\,\,\,\,2314\leq q\leq 2341.$$
	\item  $H_{i}(X,\Bbb Z_{(p)})\cong \Z_{(p)}$, for $i\in\{116,152,202,304,404,2314\}$.
\end{enumerate}	
\end{theorem}
\begin{proof}
	First, by  Theorem  \ref{t5},  to the group $G$ corresponds a strongly connected digraph  $\G$ such that $\aut(\G)\cong G$. Next,  to the  graph $\G$, we can  assign the DGL  $\mathcal {L}(\G,27)$. Then, by Theorem \ref{t2}, we get $\mathcal{E}(\mathcal {L}(\G,27))\cong\aut(\G)$. Finally,  using the Anick’s  $\Z_{(p)}$-local homotopy theory framework, to  $\mathcal {L}(\G,27)$ corresponds  an object  $X$ in the category $\textbf{DGL}^{2341}
	_{115}(\Bbb Z_{(p)})$ satisfying 
	$\mathcal{E}(X)\cong\mathcal{E}(\mathcal {L}(\G,27))$ according to  the identifications (\ref{2}). Consequently,
	 $$\mathcal{E}(X)\cong\mathcal{E}(\mathcal {L}(\G,27))\cong\aut(\G)\cong G,$$
	 as desired.
\end{proof}
\begin{remark}
By virtue of Anick's theory, the generators of the DGL $\mathcal {L}(\G,27)$ are in correspondence with the   $\Z_{(p)}$-localised  cells	of the  CW-complex $X$ constructed in Theorem \ref{01}. Thus, $X$ is finite if and only if the  group $G$ is finite.  
\end{remark}
\section*{Acknowledgments}
The author expresses sincere gratitude to the referee for their meticulous reading of the paper and for providing valuable suggestions that have significantly enhanced the manuscript.\\



\bibliographystyle{amsplain}

\end{document}